\documentclass[a4paper]{amsart}

\usepackage{amsfonts,amssymb}

\usepackage[totalwidth=14cm,totalheight=24cm]{geometry}

\usepackage[all]{xy}
\usepackage{caption}

\theoremstyle{definition}

\newtheorem{theorem}{Theorem}[section]
\newtheorem{proposition}[theorem]{Proposition}
\newtheorem{lemma}[theorem]{Lemma}
\newtheorem{corollary}[theorem]{Corollary}

\newtheorem{definition}[theorem]{Definition}
\newtheorem{example}{Example}
\newtheorem{remark}[theorem]{Remark}
\newtheorem*{notation}{Notation}

\newtheorem{question}{Question}
\newtheorem{problem}{Problem}

\def\floor#1{\left\lfloor{#1}\right\rfloor}

\def\Aut{\mathrm{Aut}}
\def\PSL{\mathrm{PSL}}
\def\RR{\mathbb{R}}
\def\CC{\mathbb{C}}
\def\HH{\mathbb{H}}
\def\NN{\mathbb{N}}
\def\ZZ{\mathbb{Z}}

\def\Gcal{\mathcal{G}}
\def\Ocal{\mathcal{O}}
\def\KK{\mathbb{K}}

\def\th{\vartheta}
\def\SL{\mathrm{SL}}
\def\gl{\mathfrak{gl}}
\def\GL{\mathrm{GL}}
\def\PGL{\mathrm{PGL}}
\def\PSU{\mathrm{PSU}}
\def\U{\mathrm{U}}
\def\SU{\mathrm{SU}}
\def\SO{\mathrm{SO}}
\def\psl{\mathfrak{sl}}
\def\u{\mathfrak{u}}
\def\tr{\mathrm{Tr}}

\def\Sym{\mathrm{Sym}}

\def\Ycal{\mathcal{Y}}
\def\Rep{\mathrm{Rep}}
\def\Higgs{\mathrm{Higgs}}

\def\Hom{\mathrm{Hom}}

\def\rk{\mathrm{rk}}

\def\End{\mathrm{End}}
\def\ord{\mathrm{ord}}

\def\Res{\mathrm{Res}}

\def\IM{\mathrm{Im}}
\def\id{\mathrm{Id}}

\def\bm#1{\text{\boldmath${#1}$}}
\def\ol#1{\overline{#1}}
\def\os#1#2{\overset{#1}{#2}}
\def\ul#1{\underline{#1}}
\def\ti#1{\tilde{#1}}
\def\wti#1{\widetilde{#1}}

\def\e{\varepsilon}

\def\co{{c}}

\def\Co{\bm{\co}}
\def\bdr#1{{#1}}

\def\bco{\bdr{\co}}
\def\bCo{\bdr{\Co}}

\def\coa{\mathfrak{c}}
\def\Coa{\bm{\coa}}
\def\bcoa{\bdr{\coa}}
\def\bCoa{\bdr{\Coa}}

\def\RRR#1{\os{\frown}{#1}}
\def\RRep{\RRR{\Rep}}
\def\Rhol{\RRR{\hol}}
\def\RRHiggs{\RRR{\Higgs}}
\def\Bun{\mathrm{Bun}}
\def\Ecal{\mathcal{E}}
\def\Vcal{\mathcal{V}}

\def\ev{\mathrm{ev}}
\def\Eig{\mathrm{Eig}}

\def\lra{\longrightarrow}
\def\rar{\rightarrow}

\def\Ad{\mathrm{Ad}}

\def\hol{\mathrm{hol}}
\def\Pic{\mathrm{Pic}}

\def\even{^{ev}}
\def\odd{^{odd}}

\def\dev{\mathrm{dev}}

\def\gfrak{\mathfrak{g}}

\def\PP{\mathbb{P}}

\def\pa{\partial}

\def\Ad{\mathrm{Ad}}

\def\hra{\hookrightarrow}
\def\arr#1{\stackrel{#1}{\lra}}

\def\Gr{\mathrm{Gr}}

\def\Det{\mathcal{D}}

\def\Eu{\mathrm{eu}}

\def\Y{\mathcal{Y}}
\def\0{\bm{0}}
\def\Iso{\mathrm{Iso}}

\def\rot{\mathrm{rot}}
\def\dis{\displaystyle}
\def\Flat{\mathrm{Fl}}
\def\Qcal{\mathcal{Q}}
\def\Xcal{\mathcal{X}}
\def\Rcal{\mathcal{R}}

\def\CL#1{\mathrm{Cl}({#1})}

\def\a{\alpha}
\def\b{\beta}

\def\NO#1{\|{#1}\|_1}



\begin{document}

\title[Representations of punctured surface groups in $\PSL_2(\RR)$]{Topology of representation spaces of surface groups in $\PSL_2(\RR)$ with assigned boundary monodromy and nonzero Euler number}

\begin{abstract}
In this paper we complete the topological description
of the space of representations of the fundamental group of a punctured surface
in $\SL_2(\RR)$ with prescribed behavior at the punctures and nonzero Euler number,
following the strategy employed by Hitchin in the unpunctured case and exploiting
Hitchin-Simpson correspondence between flat bundles
and Higgs bundles in the parabolic case. This extends previous results by
Boden-Yokogawa and Nasatyr-Steer. A relevant portion of the paper is intended
to give an overview of the subject.
%
%
\end{abstract}

\author{Gabriele Mondello}
\email{mondello@mat.uniroma1.it}
\address{``Sapienza'' Universit\`a di Roma, Dipartimento di Matematica ``Guido Castelnuovo'' - piazzale Aldo Moro 5 - 00185 Roma - Italy}

\maketitle

\tableofcontents

\setlength{\parindent}{0pt}
\setlength{\parskip}{0.2cm}

\section{Introduction}

Representations $\rho:\pi_1(M)\rar G$
of the fundamental group of of a manifold $M$ inside a Lie group $G$
naturally arise as monodromies of $(G,G/H)$-geometric structures \`a la Ehresmann
\cite{ehresmann:infinitesimales} \cite{ehresmann:locales}
on $M$ (see also \cite{goldman:geometric-structures} and \cite{goldman:locally-homogeneous-manifolds}). 

From a differential-geometric point of view, the datum of such a representation
is equivalent to that of a flat principal $G$-bundle on $M$,
or of a vector bundle of rank $N$ endowed with a flat connection and with monodromy in $G$,
in case $G\subset\GL_N$ is a linear group:
flatness is somehow the counterpart of the homegeneity of $G/H$.

Conversely, a way of ``understanding'' such a representation
$\rho$ is to {\it{geometrize}} it, namely to find a geometric structure on $M$
with monodromy $\rho$.
%
%
%
%
%
%
%
%
%
%


\subsection{Closed surfaces}

Let $S$ be a compact connected oriented surface of genus $g(S)\geq 2$.

\subsubsection{Hyperbolic structures}
A remarkable example of geometric structure on $S$ is given by hyperbolic structures, that is
hyperbolic metrics on $S$ up to isotopy.
Indeed, a hyperbolic metric is locally isometric to the upper half-plane $\HH^2$,
and so it induces a $(\PSL_2(\RR),\HH^2)$-structure on $S$ with monodromy representation
$\rho:\pi_1(S)\rar\PSL_2(\RR)\cong\mathrm{Isom}_+(\HH^2)$.
A result credited to Fricke-Klein \cite{fricke-klein:automorphen} (see also Vogt \cite{vogt:invariants})
states that hyperbolic structures
on $S$ are in bijective correspondence with a connected component
of the space $\Rep(S,G)$ of conjugacy classes of representations $\pi_1(S)\rar G$ with $G=\PSL_2(\RR)$.

\subsubsection{Euler number of a representation in $\PSL_2(\RR)$}
The connected components of the whole space $\Rep(S,\PSL_2(\RR))$
can be classified according to a topological
invariant of the $\RR\PP^1$-bundle over $S$ associated to each such representation $\rho$: the Euler number $\Eu(\rho)\in\ZZ$. 
The bound $|\Eu(\rho)|\leq -\chi(S)$ was proven
by Milnor \cite{milnor:inequality} and Wood \cite{wood:milnor}; then Goldman \cite{goldman:components} showed
that each admissible value corresponds exactly to a connected component of the representation space
and that monodromies of hyperbolic structures correspond
to the component with $\Eu=-\chi(S)$.
It is easy to see that $\rho:\pi_1(S)\rar\PSL_2(\RR)$
can be lifted to $\SL_2(\RR)$ if and only if $\Eu(\rho)$ is even.

\begin{remark}
More refined invariants of a representation are given by {\it{bounded}} characteristic classes.
The bounded Euler class for topological $S^1$-bundles was investigated by
Matsumoto \cite{matsumoto:foliated} and 
the analogous Toledo invariant for $G/H$ of Hermitian type
by Toledo \cite{toledo:complex}.
Bounded Euler and Toledo classes were used by
Burger-Iozzi-Wienhard \cite{burger-iozzi-wienhard:maximal} to characterize maximal representations.
\end{remark}

\subsubsection{Local study of the representation spaces}
Traces of a local study of $\Rep(S,G)$ are already in Weil \cite{weil:discrete-1} \cite{weil:discrete-2}.
A more general treatment of the tangent space at a point $[\rho]$ and the determination
of the smooth locus of $\Rep(S,G)$ can be found in
Goldman \cite{goldman:symplectic}, Lubotzky-Magid \cite{lubotzky-magid:representations}
and in the lectures notes \cite{goldman:representations-notes} by Goldman and \cite{labourie:book} by Labourie.
A deeper analysis of the singularities of such moduli space can be found in Goldman-Millson \cite{goldman-millson:deformation}.

\subsubsection{Symplectic structure on the representation space}
When $G$ is reductive,
a natural symplectic structure on the smooth locus of $\Rep(S,G)$ is defined by Atiyah-Bott \cite{atiyah-bott:yang-mills} by
using the equivalence between representations of the fundamental group of $S$ in $G$
and flat $G$-bundles on $S$. 

In the case of the Fricke-Klein component of $\Rep(S,\PSL_2(\RR))$,
such a symplectic structure was seen by Goldman \cite{goldman:symplectic} to agree with
the Hermitian pairing defined by Weil \cite{weil:modules} using Petersson's work
\cite{petersson:pairing} on automorphic forms. 
Ahlfors \cite{ahlfors:kaehler} showed that such a Weil-Petersson pairing
defines a K\"ahler form,
which is rather ubiquitous when dealing with deformations of hyperbolic structures
(see for instance \cite{wolpert:symplectic}, \cite{sozen-bonahon:weil-petersson},
\cite{bonsante-mondello-schlenker:cyclic2}).

\subsubsection{Flat unitary bundles and holomorphic bundles}
Consider the case $G=\U_N$ and fix a complex structure $I$ on $S$.
Representations of $\pi_1(S)$ in the unitary group $\U_N$ were object of a classical theorem
by Narasimhan-Seshadri \cite{narasimhan-seshadri:unitary}, in which 
a real-analytic correspondence is
established between irreducible representations $\pi_1(S)\rar\U_N$ and
stable holomorphic vector bundles of rank $N$ and degree $0$ on the Riemann surface $(S,I)$.
One direction is easy, since every flat complex bundle
is $I$-holomorphic; for the other direction, the authors show that stable bundles of degree $0$ admit a flat Hermitian metric: their argument works by continuity method;
an analytic proof of this statement was found later
by Donaldson \cite{donaldson:narasimhan-seshadri}
proving the convergence of the hermitian Yang-Mills flow,
as suggested in the fundamental work of Atiyah-Bott \cite{atiyah-bott:yang-mills}.

\subsubsection{Flat bundles and Higgs bundles}
The celebrated paper \cite{hitchin:self-duality} by Hitchin treated the case of representations in $G=\SL_2$
and established a real-analytic correspondence between irreducible $\rho:\pi_1(S)\rar \SL_2(\CC)$ and
stable Higgs bundles $(E,\Phi)$, namely holomorphic vector bundles $E\rar (S,I)$ of rank $N$
and trivial determinant endowed with a holomorphic $\End_0(E)$-valued $(1,0)$-form $\Phi$ on $(S,I)$
and subject to a suitable stability condition.
Compared to Narasimhan-Seshadri's, the correspondence is less intuitive, since the holomorphic structure on $E$ does not
agree with the underlying holomorphic structure on the flat complex vector bundle $V\rar S$ 
determined by the representation $\rho$ (coming from the fact that locally constant functions are holomorphic) but
it is twisted: the exact amount of such twisting is determined by the aid of the harmonic metric
on $V$, whose existence was shown by Donaldson \cite{donaldson:twisted}.
For $G=\GL_N$ or $G=\SL_N$,
the existence of the harmonic metric was proven (on any compact manifold) by Corlette \cite{corlette:harmonic} 
(and later Labourie \cite{labourie:harmonic})
and the correspondence (in any dimension) was proven by Simpson \cite{simpson:yang-mills},
who also clarified the general picture by showing \cite{simpson:higgs1} \cite{simpson:higgs2} that the fundamental objects to consider are 
local systems (classified by a ``Betti'' moduli space), vector bundles with a flat connection (classified by a ``de Rham'' moduli space)
and holomorphic Higgs bundles (classified by a ``Dolbeault'' moduli space) and by constructing their moduli spaces.

\subsubsection{Correspondence for $\SL_2(\RR)$}
Back to the rank $2$ case, among the many results contained in \cite{hitchin:self-duality}, Hitchin could determine
which Higgs bundles correspond to monodromies of hyperbolic metrics, thus parametrizing
Teichm\"uller space by holomorphic quadratic differentials on $(S,I)$ and making connection 
with Wolf's result \cite{wolf:teichmuller} (namely, the Higgs field in Hitchin's
work identifies to the Hopf differential of the harmonic map in Wolf's parametrization).
Moreover, the space of isomorphism classes of Higgs bundles $(E,\Phi)$ carries a natural $S^1$-action $u\cdot(E,\Phi)=(E,u\Phi)$,
which is also rather ubiquitous when dealing with harmonic maps with a two-dimensional domain
(for instance \cite{bonsante-mondello-schlenker:cyclic1}); in rank $2$, the locus fixed by the 
$(-1)$-involution $[(E,\Phi)]\leftrightarrow [(E,-\Phi)]$
is identified to the locus of unitary (if $\Phi=0$) or real (if $\Phi\neq 0$) representations.
This allows Hitchin to fully determine the topology
of the connected components of $\Rep(S,\PSL_2(\RR))$ with non-zero Euler number
as that of a complex vector bundle over a symmetric product of copies of $S$.
The real component with Euler number zero seems slightly subtler to treat,
since it contains classes of reducible representations (or, equivalently, of
strictly semi-stable Higgs bundles)
for which the correspondence does not hold.

\subsection{Surfaces with punctures}

Let $S$ be a compact connected oriented surface and let $P=\{p_1,\dots,p_n\}\subset S$
be a subset of $n$ distinct marked points. Denote by $\dot{S}$ the punctured surface
$S\setminus P$ and assume $\chi(\dot{S})<0$.

\subsubsection{Absolute and relative representation space}
The space $\Rep(\dot{S},G)$ of conjugacy classes of representations
$\rho:\pi_1(\dot{S})\rar G$ can be partitioned according to
the boundary behavior of $\rho$.
More explicitly, fix an $n$-uple $\bCo=(\bco_1,\dots,\bco_n)$ of conjugacy classes 
in $G$ and define $\Rep(\dot{S},G,\bCo)$ as the space
of conjugacy classes of representations
$\rho:\pi_1(\dot{S})\rar G$ that send a loop positively winding about the puncture $p_i$
to an element of 
$\bco_i\subset G$.

\subsubsection{Spherical and hyperbolic structures}
Similarly to the case of a closed surface,
isotopy classes of metrics of constant curvature $K$ are the easiest examples of
geometric structures on $\dot{S}$;
a standard requirement is to ask that
the completion of such metrics has either conical singularities or
geodesic boundary of finite length (or cusps, if $K<0$) at the punctures.
Monodromies of spherical structures ($K=1$) naturally 
take values in $\mathrm{PSU}_2\cong\SO_3(\RR)$ 
but can be lifted to $\SU_2$: if $S$ has genus $0$,
such liftability imposes restrictions on the angles of the conical points of a spherical metric
\cite{mondello-panov:spherical}.
Monodromies of hyperbolic structures ($K=-1$)
determine conjugacy classes of representations $\rho:\pi_1(\dot{S})\rar\PSL_2(\RR)$,
which are also liftable to $\SL_2(\RR)$.

\subsubsection{Euler number of a representation in $\PSL_2(\RR)$}
The Euler number of a representation $\rho:\pi_1(\dot{S})\rar\PSL_2(\RR)$
and a generalized Milnor-Wood inequality
$|\Eu(\rho)|\leq -\chi(\dot{S})$ are treated by Burger-Iozzi-Wienhard in \cite{burger-iozzi-wienhard:maximal}, who also show that
all values in the interval $[\chi(\dot{S}),-\chi(\dot{S})]$ are attained
and that representations $\rho$ with $\Eu(\rho)=-\chi(\dot{S})$
correspond to monodromies of complete hyperbolic metrics on $\dot{S}$.
Since $\Eu:\Rep(\dot{S},\PSL_2(\RR))\rar\RR$ is continuous and
its restriction to the locus 
$\Rep(\dot{S},\PSL_2(\RR),\bCo)$ is locally constant,
it is an invariant of the connected components of $\Rep(\dot{S},\PSL_2(\RR),\bCo)$.

\subsubsection{Local structure and Poisson structure}
Similarly to the closed case,
representations $\pi_1(\dot{S})\rar G$
(possibly with prescribed boundary values) correspond to flat $G$-bundles
with the same boundary monodromy; the deformation theory is also analogous.
If $G$ is reductive, a natural Poisson structure \cite{GHJW:poisson} can be defined on the smooth locus of $\Rep(\dot{S},G)$,
which restricts to a symplectic structure on the smooth locus of the spaces $\Rep(\dot{S},G,\bCo)$
(see for instance \cite{mondello:poisson} \cite{mondello:weil-petersson} for 
its link with Weil-Petersson structure when $G=\PSL_2(\RR)$ and
for explicit formulae in the case of surfaces with conical points or with boundary geodesics).

\subsubsection{Flat unitary bundles and holomorphic parabolic bundles}
Unitary representations $\pi_1(\dot{S})\rar\U_N$ determine a complex vector bundle
$\dot{V}\rar\dot{S}$ of rank $N$
endowed with a flat connection and a parallel Hermitian metric.
Such a vector bundle admits a canonical extension $V\rar S$
(Deligne \cite{deligne:equations}),
in such a way that the connection may have at worst simple poles at $P$
with 
eigenvalues of the residues in $[0,1)$
and the natural parallel Hermitian metric $H$ vanishes at $P$ of order in $[0,1)$.
Mehta-Seshadri \cite{mehta-seshadri:parabolic}
introduced the important notion of a parabolic structure on $V$ at $P$,
namely a filtration of the fibers of $V$ over $P$ by order of growth
with respect to $H$,
and established the analogue of Narasimhan-Seshadri's result:
for every complex structure $I$ on $S$, there is
a correspondence between irreducible flat unitary bundles on $\dot{S}$
of rank $N$ with prescribed monodromy at the punctures and
stable holomorphic bundles of rank $N$ and (parabolic) degree $0$ on $(S,I)$
with parabolic structure at $P$ of prescribed type.
As in the closed case, going from a flat bundle to a holomorphic parabolic bundle is easy;
conversely, the existence of a flat Hermitian metric on a stable holomorphic
bundle of degree $0$ with prescribed polynomial growth at the parabolic points $p_i$
was achieved in \cite{mehta-seshadri:parabolic} by continuity method,
and then proved by Biquard \cite{biquard:paraboliques} by analytic techniques.

\subsubsection{Flat bundles and parabolic Higgs bundles}
In a fundamental article \cite{simpson:harmonic}
Simpson established the correspondence between
representations $\rho:\pi_1(\dot{S})\rar \GL_N(\CC)$ with Zariski-dense image
and parabolic $I$-holomorphic vector bundles $E_\bullet$ of rank $N$ and degree $0$
endowed with a Higgs field $\Phi\in H^0(S,K_S\otimes \End(E_\bullet))$
subject to a suitable stability condition, the weights of $E_\bullet$ and the
residues of $\Phi$ at $P$ being determined by the values of $\rho$ on
peripheral loops. The real-analytic nature of Simpson's correspondence
was proven by Konno \cite{konno:construction}
and Biquard-Boalch \cite{biquard-boalch:wild}.
The case of a general algebraic reductive group $G$ was recently
treated by Biquard, Garcia-Prada and Mundet i Riera \cite{biquard:higher}.

\subsubsection{Topological study of moduli spaces of Higgs bundles}
Following Hitchin's ideas,
Boden-Yokogawa \cite{boden-yokogawa:parabolic}
analyzed some aspects of the case of $G=\SL_2(\CC)$
and in particular the Betti numbers of the moduli space using
Morse theory. Their result was then extended by
Logares \cite{logares:U(21)} to the case of $\mathrm{U}(2,1)$-Higgs bundles and by Garc{\'i}a-Prada, Gothen and Mu{\~n}oz
\cite{garcia-prada-gothen-munoz:rank3} to the $\SL_3(\CC)$ and $\GL_3(\CC)$ cases; on the other hand,
Garc{\'i}a-Prada, Logares and Mu{\~n}oz \cite{garcia-prada-logares-munoz:U(pq)} established the a Milnor-Wood inequality 
and determined the connected components of the moduli space of
$\mathrm{U}(p,q)$-Higgs bundles.

A different approach via orbifold structures was taken by
Nasatyr-Steer \cite{nasatyr-steer:orbifold}, who implemented Hitchin's ideas in rank $2$ and
also determined the topology of the relative $\SL_2(\RR)$-representation space
in the case of positive Euler number and elliptic boundary monodromy of finite order.

\subsubsection{Content of the paper}
%
%
After giving an overview of the subject and of the fundamental results of the theory mentioned
above, we focus on the topology of the real locus
of the moduli space of parabolic $\SL_2$-Higgs bundles,
and in particular on what happens as the parabolic structure
degenerates, or equivalently on the topology of
$\SL_2(\RR)$-representation spaces as some of the boundary monodromies cease to be strictly elliptic.\\

Fix a complex structure $I$ on $S$. Given conjugacy classes $\bcoa_i\subset\psl_2(\CC)$,
numbers $w_1(p_i)\in [0,\frac{1}{2}]$ and a line bundle $\Det$ on $S$,
we consider the moduli space $\Higgs^s(S,\bm{w},2,\Det,\bCoa)$ of
stable parabolic Higgs bundles $(E_\bullet,\Phi)$ on $S$ of rank $2$ 
with parabolic weights $w_1(p_1),\,1-w_1(p_i)$ (briefly, of type $\bm{w}$),
$\det(E_\bullet)\cong\Det$
and residue $\Res_{p_i}(\Phi)$ in $\bcoa_i$.

In order to avoid to introduce too much notation at this point,
we prefer not to give here complete statements of the main results
contained in this paper
but rather to list them in an informal way:
\begin{itemize}
\item[(1)]
stable parabolic Higgs bundles in
$\Higgs^s(S,\bm{w},2,\Det,\bCoa)$
that correspond to representations $\rho:\pi_1(\dot{S})\rar \SL_2(\RR)$
with $\Eu(\rho)>0$ are characterized in Lemma \ref{lemma:real} and
Theorem \ref{thm:correspondence};
\item[(2)]
a classification of the connected and the irreducible components
of the locus of $\Higgs^s(S,\bm{w},2,\Det,\bCoa)$ mentioned in (1)
and of its closure
and the determination of their topology is obtained in
Proposition \ref{prop:real} and Proposition \ref{prop:topology};
in particular, these results combine in Theorem \ref{thm:correspondence}
to show that 
the closure (for the classical topology)
of each connected component of $\Rep(\dot{S},\SL_2(\RR),\bCo)$
with $\Eu>0$ 
is homeomorphic 
either to a complex vector space or to an $H^1(S;\ZZ/2\ZZ)$-cover of a
complex vector bundle over a symmetric product of $S\setminus P_{hyp}$
with $P_{hyp}=\{p_i\in P\,|\,\text{$c_i$ hyperbolic}\}$;
\item[(3)]
again by Proposition \ref{prop:topology} and Theorem \ref{thm:correspondence},
the closure of each connected component of $\Rep(\dot{S},\PSL_2(\RR),\bCo)$
with $\Eu>0$ is homeomorphic to a complex vector bundle over
a symmetric product of $S\setminus P_{hyp}$;
\item[(4)]
the connected components
of $\Rep(\dot{S},\PSL_2(\RR),\bCo)$ 	
are classified by their Euler number by Corollary
\ref{cor:euler-components};
\item[(5)]
the connected components of $\Rep(\dot{S},\PSL_2(\RR),\bCo)$
that can host monodromies of hyperbolic structures are determined
in Proposition \ref{prop:uniformization} and their topology
is deduced in Corollary \ref{cor:uniformization}.
\end{itemize}

The precise topological description of such representation
spaces in the punctured case can be compared
to the one in the closed case by looking at
Theorem \ref{thm:rep-closed} and Theorem \ref{thm:rep-punctured}.

\begin{remark}
In the special case of all elliptic boundary monodromies,
the above results (1-4) are already in \cite{boden-yokogawa:parabolic} and,
for elliptic boundary monodromies of finite order, the same topological description
was obtained in \cite{nasatyr-steer:orbifold}.
The case of monodromy representations of hyperbolic structures
with cusps (and so maximal Euler number) was analyzed by
Biswas, Ar{\'e}s-Gastesi and Govindarajan \cite{biswas:parabolic}.
Similarly to what happens with closed surfaces,
the work of Wolf \cite{wolf:infinite-energy} on harmonic maps that ``open the node''
from nodal Riemann surfaces to smooth hyperbolic surfaces 
relates to the case of an $\SL_2(\RR)$-Higgs bundle with
imaginary residues at the punctures.
\end{remark}

As an example of a by-product of our analysis, we describe the topology of components
of the representation space that can contain monodromies of hyperbolic structure
(see also Corollary \ref{cor:uniformization}).

\begin{corollary}[Topology of uniformization irreducible components]\label{cor:intro}
Let $\rho$ be the monodromy representation of a hyperbolic metric on $\dot{S}$
of area $2\pi e>0$
whose completion has conical singularities of angles $\th_1,\dots,\th_k>0$ at $p_1,\dots,p_k$,
cusps at $p_{k+1},\dots,p_r$
and geodesic boundaries of lengths $\ell_{r+1},\dots,\ell_n> 0$.
Let $\bco_i$ be the conjugacy class of the monodromy $\rho$
about the $i$-th end of $\dot{S}$ and let 
$s_0=\#\{i\in\{1,\dots,k\}\,|\,\th_i\in 2\pi\NN_+\}$.\\
Then $[\rho]$ belongs to the irreducible
component of $\Rep(\dot{S},\PSL_2(\RR),\bCo)$
with Euler number $e=-\chi(\dot{S})-\sum_{i=1}^k \frac{\th_i}{2\pi}>0$, and such irreducible component is
real-analytically diffeomorphic to a holomorphic vector bundle of
rank $3g-3+n-m$ over $\mathrm{Sym}^{m-s_0}(S\setminus\{p_{k+1},\dots,p_n\})$,
where $m=\sum_{i=1}^k\floor{\frac{\th_i}{2\pi}}$.
Moreover, its closure $\Rep(\dot{S},\PSL_2(\RR),\ol{\bCo})$
is homeomorphic to a holomorphic vector bundle of
rank $3g-3+n-m$ over $\mathrm{Sym}^{m-s_0}(S\setminus\{p_{r+1},\dots,p_n\})$
\end{corollary}

\subsection{Structure of the paper}

In Section \ref{sec:representations} we review the definition of
representation space of the fundamental group of a punctured surface $\dot{S}$
and the basic smoothness results.
Then we recall the (Riemann-Hilbert) correspondence between representations of
the fundamental group of a surface in an algebraic  group $G$
and flat principal $G$-bundles.
In particular, we review the case of a linear group $G\subset\GL_N$
and the flat vector bundle of rank $N$ attached to a representation.
Then we analyze the case of a $\PSL_2(\RR)$-representation $\rho$ and
we discuss the Euler number of $\rho$, some well-known fundamental results
on the topology of the $\PSL_2(\RR)$-representation space and we compare our results with them.
Finally, we briefly mention monodromy representations
coming from hyperbolic structures possibly with cusps, conical singularities
and geodesic boundaries.

In Section \ref{sec:parabolic} we first recall the notion of parabolic bundle following Simpson
and we show examples 
over the disk $\Delta$, originating (\`a la Mehta-Seshadri) from representations
$\pi_1(\dot{\Delta})\rar \U_1$ and $\pi_1(\dot{\Delta})\rar\SU_2$.
Similarly, we introduce the definition of Higgs bundles and show examples
coming from $\pi_1(\dot{\Delta})\rar \GL_N(\CC)$ with $N=1,2$.
Then we recall the notion of slope stability and of moduli space of stable parabolic
Higgs bundles.
Finally, we specialize to the case of $G=\SL_2$
and we study the topology of the locus fixed by
the involution $[(E_\bullet,\Phi)]\leftrightarrow [(E_\bullet,-\Phi)]$
and that corresponds to representations in $\SL_2(\RR)$ with $\Eu>0$.

In the final Section \ref{sec:correspondence} we recall first
the main correspondence results in the theory of representations
of fundamental groups of surfaces and holomorphic (parabolic) (Higgs) bundles.
Then we illustrate how the correspondence works for $\SL_2(\RR)$ and
for $\PSL_2(\RR)$ and we describe the topology of the components
of the representation space. We finally conclude with two corollaries
about uniformization components.

%
%
%

\subsection{Acknowledgements}

The motivation for this work originated in conversations
about monodromies of hyperbolic structures with conical points with Roberto Frigerio, whom I wish to thank. 

I am grateful to Nicolas Tholozan (reporting
a conversation with Olivier Biquard) for observing that the case of compact components with non-degenerate parabolic
structure as in Corollary \ref{cor:supermaximal}
corresponds to the class of ``super-maximal'' representations studied by Deroin-Tholozan in \cite{deroin-tholozan:super-maximal}
and to Bertrand Deroin for explaining me that the analysis carried out in this paper leads to Corollary \ref{cor:geometric}.
I also thank Paul Seidel for pointing out a former discrepancy (now corrected) in the statement of Corollary \ref{cor:intro},
Johannes Huebschmann for precise remarks on the
locally semialgebraic nature of real representation spaces
and on their tangent spaces,
Peter Gothen and an anonymous referee for drawing my attention to some relevant works in the field,
and both referees for useful comments.

The author's research was partially supported by FIRB 2010 national grant ``Low-dimensional geometry
and topology'' (code: \texttt{RBFR10GHHH\_003}) and by GNSAGA INdAM group.


\subsection{Notation}
Let $S$ be compact, connected, oriented surface of genus $g(S)$ and
$P=\{p_1,\dots,p_n\}$ be a subset of distinct points of $S$.
Denote by
$\dot{S}$ the punctured surface $S\setminus P$ and assume $\chi(\dot{S})<0$.
%

Let $\pi$ be the fundamental group $\pi_1(\dot{S},b)$, where
$b\in \dot{S}$ is a base point, and fix
a universal cover $(\wti{\dot{S}},\ti{b})\rar (\dot{S},b)$
on which $\pi$ then acts by deck transformations.
%

Fix a standard set 
$\{\a_1,\b_1,\dots,\a_{g(S)},\b_{g(S)},\gamma_1,\dots,\gamma_n\}$ of generators of $\pi$, that satisfy the unique relation $[\a_1,\b_1]\cdots[\a_{g(S)},\b_{g(S)}]\gamma_1\cdots\gamma_n=\id$. In particular, $\gamma_i$ is freely homotopic
to a small loop that simply winds about the puncture $p_i$ counterclockwise.
We will also write $\pa_i$ for the conjugacy class of $\gamma_i$ in $\pi$
(or, equivalently, for the free homotopy class of $\gamma_i$ on $\dot{S}$).

We will denote by $G$ a
reductive real or complex algebraic group with finite center $Z=Z(G)$ and
by $\gfrak$ its Lie algebra.
If $\bco_i$ is a conjugacy class or a union of conjugacy classes
in $G$, then we denote by $\ol{\bco}_i$ its closure.
Similarly, if $\bcoa_i$ is a conjugacy class or a union of conjugacy classes in $\gfrak$,
we denote by $\ol{\bcoa}_i$ its closure.
We use the symbol $\bCo$ for an $n$-uple $(\bco_1,\dots,\bco_n)$
and $\bCoa$ for an $n$-uple $(\bcoa_1,\dots,\bcoa_n)$, and similarly
for their closures $\ol{\bCo}$ and $\ol{\bCoa}$.

If $\bm{r}=(r_1,\dots,r_n)$ is a string of $n$ non-negative real numbers, then
$\NO{\bm{r}}$ will denote their sum $r_1+r_2+\dots+r_n$.

\subsubsection{Convention}
We identify $\HH\subset\CC\PP^1$ via $z\mapsto [1:z]$,
so that a matrix in $\PSL_2(\RR)$ acts on $\HH$ as
\[
\left(
\begin{array}{cc}
a & b\\
c & d
\end{array}\right)\cdot z=
\frac{c+dz}{a+bz}
\]
Consider the transformations $R_\theta,T\in\PSL_2(\RR)$ defined as
\[
R_\theta=\left(
\begin{array}{cc}
\cos(\theta/2) & -\sin(\theta/2)\\
\sin(\theta/2) & \cos(\theta/2)
\end{array}\right), 
\quad
T=\left(
\begin{array}{cc}
1 & 0\\
1 & 1
\end{array}\right)
\]
Isometries of $\HH$ conjugate to $R_\theta$
(resp. to $T$, or to $T^{-1}$) in $\PSL_2(\RR)$
are called {\it{rotations of angle $\theta$}}
(resp. {\it{positive unipotents}}, or {\it{negative unipotents}}).

\section{Representations of the fundamental group and flat bundles}\label{sec:representations}

\subsection{Representation spaces}

The set $\Hom(\pi,G)$ of homomorphisms $\pi\rar G$ is denoted by
$\Rep_b(\dot{S},G)$ and it can be identified to
the locus
\[
\left\{
(A_1,B_1,\dots,A_{g(S)},B_{g(S)},C_1,\dots,C_n)\in G^{2{g(S)}+n}\,|\,
[A_1,B_1]\cdots[A_{g(S)},B_{g(S)}]C_1\cdots C_n=\id
\right\}
\]
inside $G^{2{g(S)}+n}$.
%
The algebraic structure on $\Rep_b(\dot{S},G)$ induced as a hypersurface in $G^{2{g(S)}+n}$
is independent of the choice of the generators. 


The group $G$ acts on $\Rep_b(\dot{S},G)$ by conjugation,
that is sending a homomorphism $\rho\in\Rep_b(\dot{S},G)$ to $\Ad_g\circ \rho$.
%
Given any other base-point $b'\in\dot{S}$,
the isomorphism $\Rep_b(\dot{S},G)\cong\Rep_{b'}(\dot{S},G)$ depends on the choice of a path
between $b$ and $b'$, but it becomes canonical after factoring out the action of $G$.

\begin{definition}
The {\it{representation space}} $\Rep(\dot{S},G)$ is the Hausdorffization of the topological quotient
$\RRep(\dot{S},G):=\Rep_b(\dot{S},G)/G$.
\end{definition}

If $G$ is compact, then all $G$-orbits in $\Rep_b(\dot{S},G)$
are closed and so the map $\RRep(\dot{S},G)\rar\Rep(\dot{S},G)$
is a homeomorphism.

\begin{remark}
If $G$ is a complex group, then $\Rep(\dot{S},G)$ agrees with the
GIT quotient $\Rep_b(\dot{S},G)/\!/G$.
If $G$ is real and $G_\CC$ is its complexification, then
$\Rep(\dot{S},G)$ is a locally semialgebraic subset
(i.e. locally defined by real algebraic equalities and inequalities)
inside the real locus $\Rep(\dot{S},G_\CC)(\RR)$,
see for instance \cite{huebschmann:singularities}.
\end{remark}

%
%

\subsubsection{Closed case}

For closed surfaces ($n=0$ and so $\dot{S}=S$), the following result was proven
by Rapinchuk, Benyash-Krivetz and Chernousov \cite{RBC:representation}.

\begin{theorem}[Irreducibility for representations in $\SL_N(\CC)$ and $\PSL_N(\CC)$]
The algebraic varieties $\Rep_b(S,G)$
are irreducible for $G=\SL_N(\CC),\, \PSL_N(\CC)$.
\end{theorem}


\begin{problem}[Topology of representation spaces]
Assuming $n=0$,
determine the topology of $\Rep_b(S,G)$ and $\Rep(S,G)$:
in particular, enumerate the connected components (for the classical topology) of
$\Rep(S,G)$.
\end{problem}

We will see below that the above problem was almost completely solved by Hitchin
\cite{hitchin:self-duality} for $G=\SL_2(\RR),\,\PSL_2(\RR)$.

\subsubsection{Punctured case}

Assume now $n>0$. Since $\pi_1(\dot{S})$ is free on $2g(S)+n-1$ generators, $\Rep_b(\dot{S},G)$ is
isomorphic to $G^{2{g(S)}+n-1}$ and 
$G$ acts diagonally by conjugation
on each factor of $G^{2{g(S)}+n-1}$.

The situation becomes more interesting if we consider the relative situation.

\begin{definition}
Let $\bCo=(\bco_1,\dots,\bco_n)$, where $\bco_i$ is a conjugacy class in $G$
(or possibly a finite union of conjugacy classes).
The {\it{relative homomorphism space}}
$\Rep_b(\dot{S},G,\bCo)$ is the locus in 
$\Rep_b(\dot{S},G)$ of homomorphisms $\rho:\pi\rar G$
such that $\rho(\pa_i)\in\bco_i$.
\end{definition}
 
Equivalently, $\Rep_b(\dot{S},G,\bCo)$ is the preimage of
$\bCo\subset G^n$ under the evaluation map $\ev:\Rep_b(\dot{S},G)\rar G^n$
that sends $\rho$ to $\big(\rho(\gamma_1),\dots,\rho(\gamma_n)\big)$.

If $G$ is a complex group, then conjugacy classes and their closures are algebraic subvarieties and so $\Rep_b(\dot{S},G,\bCo)$ is an algebraic subvariety of $\Rep_b(\dot{S},G)$. If $G$ is real, then conjugacy classes and their closures are in general
semialgebraic subsets and so
$\Rep_b(\dot{S},G,\bCo)$ is semialgebraic inside $\Rep_b(\dot{S},G)$.

\begin{example}
In the special case of $G=\PSL_2(\RR)$ conjugacy classes consisting
of hyperbolic or elliptic elements are algebraic and closed; on the other hand,
the class $\bco\subset\PSL_2(\RR)$ consisting of positive 
unipotent elements (or negative unipotent elements) is semi-algebraic and not closed:
indeed, its classical closure $\ol{\bco}$ consists of $\bco\cup\{\id\}$ and its Zariski closure
is $\bco\cup\{\id\}\cup(-\bco)$.
\end{example}

%
%

The definition of $\Rep_b(\dot{S},G,\bCo)$ as well as 
the induced (semi-)algebraic structure 
are independent of the
choice of the loops $\gamma_1,\dots,\gamma_n$.


\begin{definition}
The {\it{relative representation space}} $\Rep(\dot{S},G,\bCo)$ is the Hausdorffization of the topological quotient
$\RRep(\dot{S},G,\bCo):=\Rep_b(\dot{S},G,\bCo)/G$.
\end{definition}

As in the absolute case,
if $G$ is a complex group, then $\Rep(\dot{S},G,\bCo)$ agrees with the GIT quotient $\Rep_b(\dot{S},G,\bCo)/\!/G$
and it is an algebraic subvariety of $\Rep(\dot{S},G)$. 
If $G$ is real, then
$\Rep(\dot{S},G,\bCo)$ is a locally semialgebraic subset
of $\Rep(\dot{S},G)$.

%
%
%

\begin{problem}[Topology of relative representation spaces]
Assuming $n>0$,
determine the topology of 
$\Rep_b(\dot{S},G,\bCo)$ and $\Rep(\dot{S},G,\bCo)$: in particular,
enumerate
the connected components (for the classical topology) of $\Rep(\dot{S},G,\bCo)$.
\end{problem}

Same remarks and questions hold for representations with boundary values in $\ol{\bCo}$.


\begin{remark}\label{rmk:center}
In light of the short exact sequence $0\rar Z\rar G\rar G/Z\rar 0$,
we can view $\Rep_b(\dot{S},G)$ as an $H^1(\dot{S},Z)$-bundle over 
$\Rep_b(\dot{S},G/Z)$.
It is easy to see that, if $\tilde{c}_i\subset G$ is the preimage of the conjugacy class
$c_i\subset G/Z$ under the projection $G\rar G/Z$, then
$\Rep_b(\dot{S},G,\bm{\wti{\bCo}})$ is an $H^1(S,Z)$-bundle
over $\Rep_b(\dot{S},G/Z,\bCo)$.
\end{remark}

\subsection{Flat $G$-bundles}

Let $\Gcal$ be the sheaf of smooth functions
with values in $G$ and 
$\ul{G}$ the subsheaf of locally constant functions.

It is well-known that there is a bijective
correspondence between
$G$-local systems
on $\dot{S}$
and
principal $G$-bundles $\xi\rar\dot{S}$
endowed with a flat connection $\nabla\in \Omega^1(\xi,\gfrak)^G$.
Indeed,
for every flat $G$-bundle $(\xi,\nabla)$,
the sheaf $\ul{\xi}$ of parallel sections of $\xi$ is a local system;
vice versa, given a $G$-local system $\ul{\xi}$,
the $G$-bundle whose sheaf of smooth sections is $\xi=\Gcal\times_{\ul{G}}\ul{\xi}$ 
is endowed with a flat connection induced by
the exterior differential $d:\Ocal\rar\Omega^1$.
Such a construction also establishes
a correspondence between {\it{framed flat principal $G$-bundles}}
$(\xi,\nabla,\tau)$ consisting of a flat $G$-bundle $(\xi,\nabla)$ on $\dot{S}$
with a trivialization $\tau:\xi_b\arr{\sim} G$ at the base point $b\in\dot{S}$
and {\it{framed $G$-local systems}} $(\ul{\xi},\tau')$ consisting of
a $G$-local system $\ul{\xi}$ with a trivialization
$\tau':\ul{\xi}_b\arr{\sim} G$ at $b$.
The isomorphisms $\tau$ and $\tau'$ are called {\it{framings}}.

\begin{notation}
We will denote by $\Flat(\dot{S},G)$ the set isomorphism classes of flat
principal $G$-bundles
on $\dot{S}$ and by
$\Flat_b(\dot{S},G)$ the set of isomorphism classes of $b$-framed flat principal
$G$-bundles on $\dot{S}$.
\end{notation}

We wish to recall the well-known correspondence between
flat $G$-bundles and $G$-representations of $\pi$ and to adapt it
to the framed case.

Consider the trivial $G$-bundle $\ti{\xi}:=\wti{\dot{S}}\times G\rar\wti{\dot{S}}$
with the framing $\ti{\tau}: \ti{\xi}_{\ti{b}}=\{\ti{b}\}\times G\arr{\sim}G$
given by the projection onto the second factor.
The flat connection $\wti{\nabla}$ on $\ti{\xi}$ is simply given by the de Rham
differential.

Given a representation $\rho\in\Rep_b(\dot{S},G)$,
the fundamental group $\pi$ acts on $\ti{\xi}$, and more precisely via deck transformations
on the factor $\wti{\dot{S}}$ and via $m_L\circ \rho$ on the factor $G$, where
$m_L$ is the action of $G$ on $G$ by left multiplication.
As a consequence, $\wti{\nabla}$ descends to a flat connection $\nabla$ on
the $G$-bundle $\xi_\rho:=\ti{\xi}/\pi$ on $\dot{S}$.
Moreover, $\ti{\tau}$ induces a framing $\tau:\xi_b\arr{\sim}G$ through
the isomorphism $\ti{\xi}_{\ti{b}}\arr{\sim}\xi_b$. This construction determines
an application
\[
\Xi_b:\Rep_b(\dot{S},G)\lra \Flat_b(\dot{S},G).
\]
Vice versa, given a framed flat $G$-bundle $(\xi,\nabla,\tau)$,
the holonomy representation based at $b$ descends
to a homomorphism $\pi\rar \Aut(\xi_b)\cong G$ by the flatness of $\nabla$,
and so via $\tau$ to a homomorphism $\rho:=\hol_b(\xi):\pi\rar G$.
This construction determines an application
\[
\hol_b:\Flat_b(\dot{S},G)\lra
\Rep_b(\dot{S},G).
\]
It is easy to check that $\Xi_b$ and $\hol_b$ are set-theoreticallly inverse of each other.

Any two trivializations $\tau_1,\tau_2: \xi_b\rar G$ at $b$ are related by 
a unique element $g\in G$, namely $\tau_2=m_L(g)\circ\tau_1$.
Hence, factoring out the action of $G$ one obtains the applications
$\Xi: \RRep(\pi,G)\rar \Flat(\dot{S},G)$ and
$\Rhol: \Flat(\dot{S},G)\rar \RRep(\pi,G)$
which
are set-theoretically inverse of each other.

\begin{remark}
Viewing $\Flat_b(\dot{S},G)$ as a quotient of the space of flat connections
on $G\times \dot{S}\rar \dot{S}$, it is possible to endow $\Flat_b(\dot{S},G)$
with the structure of (real or complex) analytic variety that makes
$\hol_b$ and $\Xi_b$ analytic isomorphisms. Since we do not want to go
deeper in this direction, we can alternatively just put on $\Flat_b(\dot{S},G)$
the analytic structure induced by $\hol_b$ and $\Xi_b$.
\end{remark}

The correspondence in the relative case is dealt with analogously.

\begin{notation}
Denote by $\Flat_b(\dot{S},G,\bCo)$
the set of isomorphism classes of $b$-framed flat $G$-bundles on $\dot{S}$
with holonomy along the path $\pa_i$ in $\co_i$,
and by $\Flat(\dot{S},G,\bCo)$ the
set of isomorphisms classes of $G$-bundles on $\dot{S}$ with holonomy
along the $i$-th end in $\co_i$.
Analogous notation for flat bundles with boundary monodromy in $\ol{\bCo}$.
\end{notation}

The same construction as above works in the relative case and
we summarize the discussion in the following statement.

\begin{lemma}[Equivalence between representations of $\pi_1$ and flat $G$-bundles]
The applications
\[
\xymatrix@R=0in{
\Rep_b(\dot{S},G,\bCo)
\ar@/^1pc/[rr]^{\hol_b} && \ar@/^1pc/[ll]^{\Xi_b} 
\Flat_b(\dot{S},G,\bCo)
}
\]
are set-theoretically inverse of each other.
By factoring the $G$-action by conjugation on the representation
and by left multiplication on the $b$-framing, we recover
the correspondence
\[
\xymatrix@R=0in{
\RRep(\dot{S},G,\bCo)\ar@/^1pc/[rr]^{\Rhol} && \ar@/^1pc/[ll]^{\Xi} 
\Flat(\dot{S},G,\bCo).
}
\]
\end{lemma}

We call the composition $\Flat(\dot{S},G,\bCo)\arr{\Rhol}\RRep(\dot{S},G,\bCo)\rar \Rep(\dot{S},G,\bCo)$ simply $\hol$.

\subsection{Smoothness}

Given a representation $\rho:\pi\rar G$,
we denote by $\gfrak^\rho\subset\gfrak$ (resp. $G^\rho\subset G$) 
the subset of elements which are invariant
under the adjoint action of $\pi$ though $\rho$.
 
We can naturally identify
$\Aut(\xi_\rho)$ with the centralizer 
$Z_G(\rho(\pi))=G^{\rho}$ of $\rho(\pi)$ in $G$
and
the space $T_\id \Aut(\xi_\rho)$ of first-order automorphisms
with $\gfrak^\rho\cong H^0(\dot{S},\Ad(\xi_\rho))$, where
$\Ad(\xi_\rho)$ is the flat $\gfrak$-bundle with monodromy $\Ad\circ\rho$.


\begin{definition}
A representation $\rho$ is {\it{regular}} if $\gfrak^\rho=\{0\}$; it is
{\it{Zariski-dense}} if its image is Zariski-dense.
We denote by $\Rep_b^r(\dot{S},G)$ (resp. by $\Rep_b^{Zd}(\dot{S},G)$) the
subsets of $\Rep_b(\dot{S},G)$ of regular (resp. Zariski-dense) representations,
and similarly by $\Rep^r(\dot{S},G)$ (resp. by $\Rep^{Zd}(\dot{S},G)$) the corresponding
locus in $\Rep(\dot{S},G)$.
\end{definition}

Since we are assuming $G$ algebraic, $\rho$ is regular
if and only if $Z(\rho(\pi))$ is finite. 
%
%
A proof of the following statement can be found in \cite{johnson-millson:deformation} and
in Section 5.3.4 of \cite{labourie:book}.

\begin{lemma}[Proper action of $G$ on $\Rep_b^r$]\label{lemma:properness}
The action on the subset $\Rep_b^{Zd}(\dot{S},G)\subseteq\Rep_b^r(\dot{S},G)$ is proper and with stabilizer $Z$.
\end{lemma}

A standard deformation theory argument shows that
first-order deformations of $\xi$ are parametrized by $H^1(\dot{S},\Ad(\xi))$.
The proof of the following result can be found for instance in \cite{goldman:symplectic}, \cite{johnson-millson:deformation},
\cite{labourie:book}.

\begin{proposition}[Tangent space to representation spaces]\label{prop:smoothness}
For every representation $\rho\in\Rep_b(\dot{S},G)$, we have
\[
H^0(\dot{S};\Ad(\xi_\rho))  \cong \gfrak^\rho 
\qquad\text{and}\qquad
H^2(\dot{S};\Ad(\xi_\rho))  \cong
\begin{cases} 
(\gfrak^\vee)^\rho & \text{if $n=0$}\\ 
0 & \text{otherwise}
\end{cases}
\]
by Poincar\'e duality. The centralizer $G^\rho$ acts on 
$H^1(\dot{S};\Ad(\xi_\rho))$ by adjunction and the quotient
can be identified to $T_{[\rho]}\Rep(\dot{S},G)$.
%
Hence, the regular locus is an open dense orbifold in the representation spaces and
\begin{align*}
\def\arraystretch{1.3}\begin{array}{|c|c|c|}
\hline
& \Rep^r_b(\dot{S},G) & \Rep^r(\dot{S},G)\\
\hline
\dim & (-\chi(\dot{S})+1)\dim(G) 
 & -\chi(\dot{S})\dim(G) 
 \\
\hline
\end{array}
\end{align*}
for any $n\geq 0$.
%
%
%
\end{proposition}

%

The tangent space $T_{[\rho]}\Rep(\dot{S},G,\bCo)$
in the relative case can be analyzed by means of the
following exact sequence associated to the couple $(\dot{S},\dot{\Delta})$
\begin{align*}
0=H^0(\dot{S},\dot{\Delta};\Ad(\xi))\rar H^0(\dot{S};\Ad(\xi))
\rar H^0(\dot{\Delta};\Ad(\xi))\rar H^1(\dot{S},\dot{\Delta};\Ad(\xi))\rar\\ 
\rar H^1(\dot{S};\Ad(\xi))\rar H^1(\dot{\Delta};\Ad(\xi))
\rar H^2(\dot{S},\dot{\Delta};\Ad(\xi))\rar H^2(\dot{S};\Ad(\xi))=0
\end{align*}
where
the $\Delta_i$'s are disjoint open contractible neighbourhoods of the $p_i$'s
and
$\dot{\Delta}=\bigcup_{i=1}^n \dot{\Delta}_i$
is the union of the punctured disks $\dot{\Delta}_i=\Delta_i\setminus\{p_i\}$.

In fact, since $G$ is reductive, the Lie algebra $\gfrak$ has a non-degenerate invariant symmetric bilinear form, 
which induces an $\Ad$-invariant isomorphism $\gfrak\cong\gfrak^\vee$. By Lefschetz duality, 
$H^2(\dot{S},\dot{\Delta};\Ad(\xi))\cong H^0(\dot{S};\Ad(\xi))^\vee$ and 
$H^1(\dot{S},\dot{\Delta};\Ad(\xi))\cong H^1(\dot{S};\Ad(\xi))^\vee$.
Thus, we obtain the $G^\rho$-equivariant exact sequence
\[
0\rar \gfrak^\rho\rar H^0(\dot{\Delta};\Ad(\xi))\arr{\epsilon^\vee}
H^1(\dot{S};\Ad(\xi))^\vee \rar   H^1(\dot{S};\Ad(\xi))  \arr{\epsilon}
H^1(\dot{\Delta};\Ad(\xi))\rar (\gfrak^\vee)^\rho \rar 0
\]
Cocycles in $\ker(\epsilon)$ induce first-order deformations
$\rho_t$ of $\rho$ such that
for every $i$ the boundary value $\rho_t(\gamma_i)$ is conjugate to $\rho(\gamma_i)$ for all $t$.
Thus, taking $G^\rho$-coinvariants, we obtain
$T_{[\rho]}\Rep(\dot{S},G,\bCo)\cong \ker(\epsilon)_{G^\rho}$.

Combining the above computation with the properness in
Lemma \ref{lemma:properness}, we can conclude as follows.

\begin{corollary}[Tangent space to relative representation spaces]\label{cor:relative-smoothness}
For $n>0$, the locus $\Rep^r(\dot{S},G,\bCo)$ is a smooth orbifold of
dimension 
\[
\def\arraystretch{1.3}\begin{array}{|c|c|c|}
\hline
& \Rep^r_b(\dot{S},G,\bCo) & \Rep^r(\dot{S},G,\bCo)\\
\hline
\dim & (-\chi(S)+1)\dim(G)+ \dim(\bCo) & 
-\chi(S)\dim(G) 
+\dim(\bCo)\\
\hline
\end{array}
\]
Moreover, the singular locus of $\Rep^r(\dot{S},G,\ol{\bCo})$
consists of those $[\rho]$ with boundary values
in the singular locus of $\ol{\bCo}$.
\end{corollary}


\subsection{Flat vector bundles}

Let $\KK=\RR,\CC$ and consider first $G=\GL_N(\KK)$.
To every flat principal $G$-bundle $\xi$ on $\dot{S}$
we can associate a vector bundle 
$V=\xi\times_G \KK^N$ of rank $N$ endowed with a natural flat connection $\nabla$
in such a way that
the monodromy of $(V,\nabla)$ coincides with that of $\xi$.

Vice versa, given a flat vector bundle $(V,\nabla)$, we can construct
the associate flat principal $G$-bundle $\xi$ using the same locally constant
transition functions as $V$, so that $\xi$ has the same monodromy as $V$.

This establishes a correspondence
\[
\xymatrix@R=0in{
\Flat(\dot{S},\KK^N)\ar@/^0.5pc/[rr] && \ar@/^0.5pc/[ll] 
\Flat(\dot{S},\GL_N(\KK)).
}
\]
where $\Flat(\dot{S},\KK^N)$ is the set of isomorphism classes of flat vector bundles of rank $N$.
An analogous correspondence holds for framed flat bundle, or for flat bundles with monodromy at the punctures in prescribed conjugacy classes.

For $G=\SL_N(\KK)$, the correspondence is between
flat principal $\SL_N$-bundles and flat vector bundles $V$ of rank $N$
endowed with a trivialization of their determinant $\det(V)=\Lambda^N V$.

Similarly, flat $\PGL_N$-bundles correspond to flat $\KK\PP^{N-1}$-bundles.
Such a $\KK\PP^{N-1}$-bundle $\PP$ need not be a projectivization of a flat
vector bundle, since its monodromy need not lift to $\GL_N(\KK)$.

In the real case, $\PSL_{2N+1}(\RR)=\PGL_{2N+1}(\RR)$; whereas
flat $\PSL_{2N}(\RR)$-bundles correspond
to flat orientable $\RR\PP^{2N-1}$-bundles.

\subsection{The case of $\PSL_2$}

Assume now that $G/Z=\PSL_2(\CC)$, namely that $G=\SL_2(\CC),\PSL_2(\CC)$,
and so $\gfrak=\mathfrak{sl}_2(\CC)$.
Then to each flat principal $G$-bundle $\xi\rar\dot{S}$ we can associate
a flat $\CC\PP^1$-bundle $\PP:=\xi\times_G \CC\PP^1$ on $\dot{S}$.

\begin{remark}
The image of $\rho$ is Zariski-dense if and only if
$\rho$ is irreducible, namely
if and only if no point of $\CC\PP^1$ is fixed by $\rho(\pi)$. 
The same holds for
$G/Z=\PSU_2\subset \PSL_2(\CC)$. In case $G/Z=\PSL_2(\RR)$, such Zariski-density
can be expressed in terms of non-existence of fixed points in $\ol{\HH}$.
\end{remark}


An easy consequence of Corollary \ref{cor:relative-smoothness}
is the following.



\begin{corollary}[Smoothness for $\PSL_2$]\label{cor:smoothness-psl2}
Let $G/Z$ be $\PSL_2(\CC),\ \PSU_2,\ \PSL_2(\RR)$ and let
$\bCo=(\co_1,\dots,\co_n)$ an $n$-uple of conjugacy classes of elements of $G$.
\begin{itemize}
\item[(a)]
If $G/Z=\PSU_2$, then 
$\Rep_b^{r}(\dot{S},G,\bCo)$
and
$\Rep^{r}(\dot{S},G,\bCo)$ 
are smooth orbifolds.
\item[(b)]
If $G/Z=\PSL_2(\CC),\PSL_2(\RR)$, then 
$\rho\in \Rep_b^{r}(\dot{S},G,\ol{\bCo})$
and $[\rho]\in\Rep^{r}(\dot{S},G,\ol{\bCo})$
are singular points if and only if
there exists $i$ such that
$\rho(\pa_i)\in Z$ and $\co_i$ consists of
non-central unipotent elements.
%
\end{itemize}
\end{corollary}

Condition (b) in the above lemma can be easily understood by remembering
that the only non-closed conjugacy class in $\PSL_2(\CC)$ is
the class $\co$ of non-trivial unipotent elements, whose closure
contains the identity as a singular point of $\ol{\co}$.


\subsection{Euler number of $\PSL_2(\RR)$-representations}

Let $\wti{\PSL}_2(\RR)$ be the universal cover of $\PSL_2(\RR)$
and let $\ZZ \cdot\zeta\subset\wti{\PSL}_2(\RR)$ be its center,
where $\zeta=\exp(R)$ and
\[
R=\pi\left(
\begin{array}{cc}
0 & -1\\
1 & 0
\end{array}
\right)\in\psl_2(\RR)
\]
is an infinitesimal generator of the subgroup of
(counterclockwise) rotations that fix $i\in\HH$.
We define the {\it{rotation number}} $\rot:\wti{\PSL}_2(\RR)\rar\RR$
as follows.

If $\ti{g}\in\wti{\PSL}_2(\RR)$ is elliptic, then
$\ti{g}$ is conjugate to $\exp(r\cdot R)$ for a unique $r\in\RR$.
In this case, we define $\rot(\ti{g}):=r$,
so that $\rot(\ti{g})\in \ZZ \iff \ti{g}\in \ZZ\cdot\zeta$.

If $\ti{g}\in\wti{\PSL}_2(\RR)$ is not elliptic,
then there exists a unique $r\in\ZZ$ such that
$\ti{g}$ can be connected to $r\cdot\zeta$ through a
continuous path of non-elliptic elements. In this case,
we define $\rot(\ti{g}):=r$.

Define also a ``fractional'' rotation number as
$\{\rot\}:\PSL_2(\RR)\rar [0,1)$
by requiring that $\rot(\ti{g})-\{\rot\}(g)\in\ZZ$,
where $\ti{g}$ is any lift of $g$ to $\wti{\PSL}_2(\RR)$.

It can be easily seen that $\rot$  and $\{\rot\}$
are invariant under conjugation by elements of $\PSL_2(\RR)$
and that
the rotation number
$\rot$ is continuous but $\{\rot\}$ is not (see also \cite{burger-iozzi-wienhard:maximal} for
more properties of the rotation number).

Here we adopt a result by Burger-Iozzi-Wienhard \cite{burger-iozzi-wienhard:maximal}
as a definition of Euler number for a representation $\rho:\pi_1(\dot{S})\rar\PSL_2(\RR)$.

\begin{definition}
Assume $n>0$. The {\it{Euler number of $\rho:\pi\rar\PSL_2(\RR)$}} is
\[
\Eu(\rho):=-\sum_{i=1}^n r_i\in\RR
\]
where $\ti{\rho}:\pi\rar\wti{\PSL}_2(\RR)$ is any lift of $\rho$ and
$r_i:=\rot(\ti{\rho}(\gamma_i))\in \RR$.
%
%
\end{definition}

Notice that $\Eu(\rho)+\NO{\bm{\{r\}}}\in \ZZ$, where $\{r_i\}:=\{\rot\}(\rho(\gamma_i))\in[0,1)$.
%
%

\begin{remark}\label{rmk:identity}
If a representation is obtained as a composition
$\rho':\pi_1(S\setminus\{p_1,\dots,p_n\},b)\rar\pi_1(S\setminus\{p_1,\dots,p_k\},b)\arr{\rho}\PSL_2(\RR)$ with $0\leq k<n$,
then $\{r_{k+1}\}=\dots=\{r_n\}=0$ and $\Eu(\rho')=\Eu(\rho)$. This allows to coherently define the Euler number
in the case of an unpunctured surface.
\end{remark}

\begin{remark}
Suppose that $[\rho]$ is a singular point of $\Rep(\dot{S},\PSL_2(\RR))$.
Then $\rho$ must fix a point of $\RR\PP^1$ 
or it must be Abelian.
In the former case, there exists a lift $\ti{\rho}$
whose action on the universal cover $\wti{\RR\PP^1}$
fixes a point. Hence,
$\rot\circ\ti{\rho}=0$ and so
$\Eu(\rho)=0$ and $\{r_j\}=0$ for all $j$.
In the latter case, $\rho$ Abelian implies $\Eu(\rho)=0$ and so $\NO{\bm{\{r\}}}\in\ZZ$.
\end{remark}

Being invariant under conjugation by elements of $\PSL_2(\RR)$, the Euler number descends to
a continuous map
\[
\Eu:\Rep(\dot{S},\PSL_2(\RR))\lra \RR
\]
and we call $\Rep(\dot{S},\PSL_2(\RR))_e$ the preimage $\Eu^{-1}(e)$.
If $n=0$, then $\Eu$ is integral and so it is constant on connected components
of $\Rep(S,\PSL_2(\RR))$. Moreover, the conjugation by an element in $\PGL_2(\RR)\setminus\PSL_2(\RR)$ induces the isomorphism
$\Rep(S,\PSL_2(\RR))_e\cong \Rep(S,\PSL_2(\RR))_{-e}$.

For closed surfaces the topology of $\Rep(S,\PSL_2(\RR))_e$ with $e\neq 0$ is completely determined.

\begin{theorem}[Topology of representation spaces of closed surfaces in $\PSL_2(\RR)$]\label{thm:rep-closed}
Assume $n=0$ and $g(S)\geq 2$.
\begin{itemize}
\item[(a)]
Every $\rho\in\Rep_b(S,\PSL_2(\RR))$ satisfies $|\Eu(\rho)|\leq -\chi(S)$ (Milnor \cite{milnor:inequality}, Wood \cite{wood:milnor}).
If $\Eu(\rho)=-\chi(S)$, then $\rho$ is the monodromy of a hyperbolic metric (Goldman \cite{goldman:components}).
\item[(b)]
$\Rep(S,\PSL_2(\RR))_e\neq\emptyset$
if and only if
\[
e\in \ZZ\cap \Big[ \chi(S),\,-\chi(S)\,\Big]
\]
and, in this case, it is also connected
(Goldman \cite{goldman:components}).
\item[(c)]
If $e\neq 0$, then $\Rep(S,\PSL_2(\RR))_e$ is smooth.\\
For $e>0$ the manifold $\Rep(S,\PSL_2(\RR))_e$ is real-analytically diffeomorphic to
a complex vector bundle of rank $-\frac{3}{2}\chi(S)-m$ over $\Sym^m(S)$,
where $m=-\chi(S)-e$ (Hitchin \cite{hitchin:self-duality}).
\end{itemize}
\end{theorem}

%

Let now $n>0$. Given an $n$-uple $\bCo=(\bco_1,\dots,\bco_n)$ of conjugacy classes
in $\PSL_2(\RR)$, we call $\{r_i\}$ the rotation number of any element in $\bco_i$.
Moreover, we denote by $P_{hyp}$ 
(resp. $P_{ell}$, $P_+$, $P_-$, $P_0$)
the subset of points $p_i\in P$ such that $\bco_i$ is hyperbolic
(resp. elliptic, positive unipotent, negative unipotent, the identity)
and we let $s_-=\# P_-$ (resp. $s_0=\# P_0$).

The restriction
\[
\Eu_{\bCo}:\Rep(\dot{S},\PSL_2(\RR),\bCo)\lra\RR
\]
of $\Eu$ has the property that $\Eu_{\bCo}+\NO{\bm{\{r\}}}\in \ZZ$, where
$\NO{\bm{\{r\}}}$ only depends on $\bCo$.
Thus, $\Eu_{\bCo}$ is also constant on connected components of
$\Rep(\dot{S},\PSL_2(\RR),\bCo)$.

As above, we denote by $\Rep(\dot{S},\PSL_2(\RR),\bCo)_e$
the preimage $\Eu_{\bCo}^{-1}(e)$ and we observe that
the conjugation
by an element of $\PGL_2(\RR)\setminus \PSL_2(\RR)$
induces the isomorphism
$\Rep(\dot{S},\PSL_2(\RR),\bCo)_e
\cong \Rep(\dot{S},\PSL_2(\RR),\bCo^{-1})_{-e}$,
where $\bco_i^{-1}=\{g^{-1}\in\PSL_2(\RR)\,|\,g\in\bco_i\}$.

In analogy with Theorem \ref{thm:rep-closed} in the closed case,
the following result holds in the punctured case.

\begin{theorem}[Topology of relative representation spaces of punctured surfaces in $\PSL_2(\RR)$]\label{thm:rep-punctured}
Assume $n>0$ and let $\bCo$ and $\bm{\{r\}}$ be as above.
\begin{itemize}
\item[(a)]
The image of $\Eu:\Rep(\dot{S},\PSL_2(\RR))\rar\RR$ is the interval $\big[\chi(\dot{S}),-\chi(\dot{S})\big]$.
If $\Eu(\rho)=-\chi(\dot{S})$, then all $\rho(\pa_i)$'s are hyperbolic or positive unipotent elements for all $i$ and $\rho$ is the monodromy of a hyperbolic metric with geodesic boundary components and cusps (Burger-Iozzi-Wienhard \cite{burger-iozzi-wienhard:maximal}).
\item[(b)]
Assume $e>0$ and fix conjugacy classes $\bco_1,\dots,\bco_n\subset\PSL_2(\RR)$.\\
Then $\Rep(\dot{S},\PSL_2(\RR),\bCo)_e\neq\emptyset$ if and only if
\[
e+\NO{\bm{\{r\}}} +s_0+s_-\in \ZZ\cap \left( 0,-\chi(\dot{S})\right]
\]
%
%
and, in this case, it is smooth and connected.
%
\end{itemize}
Assume now that 
$e$ satisfies the hypotheses in (b).
\begin{itemize}
\item[(c)]
The component $\Rep(\dot{S},\PSL_2(\RR),\bCo)_e$ is 
real-analytically diffeomorphic to the complement of 
$s_-$ affine subbundles of codimension $1$
inside a holomorphic affine bundle
of rank $-\frac{3}{2}\chi(S)+n-m+s_-$ 
over $\Sym^{m-(s_0+s_-)}\Big(S\setminus (P_{hyp}\cup P_+)\Big)$, 
where $m=-\chi(\dot{S})-e-\NO{\bm{\{r\}}}$.
\item[(d)]
The locus $\Rep(\dot{S},\PSL_2(\RR),\ol{\bCo})_e$ is homeomorphic to
an affine holomorphic bundle of rank $-\frac{3}{2}\chi(S)+n-m+s_-$
over $\Sym^{m-(s_0+s_-)}(S\setminus P_{hyp})$.
%
%
\item[(e)]
If $c_i$ is the class of positive unipotents and $\bCo^0$ is obtained
from $\bCo$ by replacing $c_i$ with $\{\id\}$, then
$\Rep(\dot{S},\PSL_2(\RR),\ol{\bCo}^0)_e$ includes
in $\Rep(\dot{S},\PSL_2(\RR),\ol{\bCo})_e$ as the preimage over
$p_i+\Sym^{m-(s_0+s_-)-1}(S\setminus P_{hyp})\subset \Sym^{m-(s_0+s_-)}(S\setminus P_{hyp})$.\\
If $c_i$ is the class of negative unipotents and $\bCo^0$ is obtained
from $\bCo$ by replacing $c_i$ with $\{\id\}$, then
$\Rep(\dot{S},\PSL_2(\RR),\ol{\bCo}^0)_e$ includes
in $\Rep(\dot{S},\PSL_2(\RR),\ol{\bCo})_e$ as an affine subbundle over
$\Sym^{m-(s_0+s_-)}(S\setminus P_{hyp})$
of codimension $1$.
%
%
%
%
%
%
%
\end{itemize}
\end{theorem}

Claims (c-d-e) are consequence of
Theorem \ref{thm:correspondence}, which in turn relies on Proposition \ref{prop:topology}
and Corollary \ref{cor:quotient-topology}. Claim (b) easily follows
from (c-d-e).

The case of some $\bco_i=\{\id\}$ can be also dealt with as in Remark \ref{rmk:identity}.

\subsection{Hyperbolic metrics}

%
%

Let $\bm{\ell}=(\ell_1,\dots,\ell_n)$
with $\ell_i=\sqrt{-1}\th_i$ and $\th_i>0$ for $i=1,\dots,k$ and $\ell_i\geq 0$ for $i=k+1,\dots,n$.
We are interested in isotopy classes of {\it{hyperbolic metrics of boundary type $\bm{\ell}$}}
on $\dot{S}$, i.e. metrics on curvature $-1$ on $\dot{S}$, whose completion has
a conical singularity of angle $\th_i$ at $p_i$ for $i=1,\dots,k$
and a boundary component of length $\ell_i$ (resp. a cusp if $\ell_i=0$)
instead of the puncture $p_i$ for $i=k+1,\dots,n$.

We assume that the quantity
\[
e_{\bm{\ell}}:=-\chi(\dot{S})-\sum_{i=1}^k\frac{\th_i}{2\pi}=
-\chi(S)-\sum_{i=1}^k\left(\frac{\th_i}{2\pi}-1\right)
\]
is positive, since 
hyperbolic metrics of boundary type $\bm{\ell}$ have total area $e_{\bm{\ell}}>0$ by Gauss-Bonnet.
In this case, denote by $\Y(\dot{S},\bm{\ell})$ the space of isotopy classes of hyperbolic metrics on $\dot{S}$ of boundary
type $\bm{\ell}$.

Surfaces of curvature $-1$ are locally isometric to portions
of the hyperbolic plane $\HH$.
Given a metric $h$ of curvature $-1$ on $\dot{S}$,
consider the pull-back $\ti{h}$ on $\ti{\dot{S}}$.
Since $\ti{\dot{S}}$ is simply-connected, local isometries into $\HH$
glue to give a global {\it{developing map}} $\dev_h:\ti{\dot{S}}\rar\HH$,
which is a local isometry.
Moreover, $\pi$ acts on $\HH$ via
a {\it{monodromy}} homomorphism $\rho_h:\pi\rar\Iso_+(\HH)\cong\PSL_2(\RR)$
and $\dev_h$ is $\pi$-equivariant.
Notice that $\dev_h$ is well-defined up to post-composition with an isometry
of $\HH$, and so also $\rho_h$ is well-defined only as an element
of $\Rep(\dot{S},\PSL_2(\RR))$.

We also observe that $\rho_h$ arises as a monodromy of a flat principal
$\PSL_2(\RR)$-bundle
as follows. 
Pull the trivial $\PSL_2(\RR)$-bundle over $\HH$ back via $\dev_h$
to a (trivializable) $\ti{\xi}_h\rar\ti{\dot{S}}$.
By $\rho_h$-equivariance, it descends to a flat principal
$\PSL_2(\RR)$-bundle $\xi_h\rar\dot{S}$ with
$\hol_{\xi_h}=\rho_h$.

%

%
%
%
%


Then we have a composition of real-analytic maps
\[
\xymatrix@R=0in{
\Y(\dot{S},\bm{\ell})\ar[r]^{\Xi_{\bm{\ell}}\qquad} 
& \Flat(\dot{S},\PSL_2(\RR),\bCo_{\bm{\ell}}) \ar[r]^{\hol\quad}
& \Rep(\dot{S},\PSL_2(\RR),\bCo_{\bm{\ell}})\\
h \ar@{|->}[r] & \xi_h \ar@{|->}[r] & \rho_h
}
\]
where the correspondence between $\ell_i$ and $\bco_i$
is dictated by the following table
\[
%
%
\def\arraystretch{1.3}\begin{array}{|c|c|}
\hline
\text{conjugacy class $\bco_\ell$} & \ell\\
\hline 
\id & 2\pi\sqrt{-1}\cdot\NN_+\\
\text{positive unipotents} & 0\\
\text{hyperbolics $g$ with $|\tr(g)|=2|\cosh(\ell/2)|$} & \RR_+\\
\text{elliptics $g$ with $\{\rot\}(g)=\{\ell/(2\pi\sqrt{-1})\}$} & 2\pi\sqrt{-1}\cdot (\RR_+\setminus\NN_+)\\
\hline
\end{array}
\]
In fact, 
a cusp at $p_i$ corresponds to positive unipotent monodromy
along $\gamma_i$.

%
%
%
%
%


\begin{proposition}[Uniformization components]\label{prop:uniformization}
Let $\th_1,\dots,\th_k>0$ and $\ell_{k+1},\dots,\ell_n\geq 0$
and call $\bm{\ell}=(\sqrt{-1}\th_1,\dots,\sqrt{-1}\th_k,\ell_{k+1},\dots,\ell_n)$.
If $e_{\bm{\ell}}>0$, then
the image of $\hol\circ\Xi_{\bm{\ell}}$ is contained inside
$\Rep(\dot{S},\PSL_2(\RR),\bCo_{\bm{\ell}})_{e_{\bm{\ell}}}$.
%
\end{proposition}
\begin{proof}
After the above discussion, it is enough to notice that
the Euler number $\Eu$ can be identified to the Toledo invariant
(see for instance \cite{burger-iozzi-wienhard:maximal})
and so $\Eu(\hol_h)=\frac{1}{2\pi}\mathrm{Area}(h)$.
\end{proof}


\subsubsection{Questions}
Still little is known in general about the image of
\[
\hol\circ\Xi_{\bm{\ell}}:\Ycal(\dot{S},\bm{\ell})\lra \Rep(\dot{S},\PSL_2(\RR),\bCo_{\bm{\ell}})_{e_{\bm{\ell}}}
\]
beside the fact that it is open for the classical topology and so its complement $\mathcal{C}_{\bm{\ell}}$ is closed.

\begin{question}\label{q:surjective}
For which $\bm{\ell}$ is the image of $\hol\circ\Xi_{\bm{\ell}}$ the whole $\Rep(\dot{S},\PSL_2(\RR),\bCo_{\bm{\ell}})_{e_{\bm{\ell}}}$?
\end{question}

\begin{question}\label{q:zero}
For which $\bm{\ell}$ is $\mathcal{C}_{\bm{\ell}}$ of zero measure?
\end{question}

\begin{question}
For which $\bm{\ell}$ is $\mathcal{C}_{\bm{\ell}}$ a countable union of proper (semi-)algebraic subsets with no internal part?
\end{question}

\begin{question}\label{q:ergodic}
For which $\bm{\ell}$ does the mapping class group of $(S,P)$ act ergodically on $\Rep(\dot{S},\PSL_2(\RR),\bCo_{\bm{\ell}})_{e_{\bm{\ell}}}$?
\end{question}

Question \ref{q:surjective} has affirmative answer if all the angles are smaller than $\pi$.
Roughly speaking, a possible argument is as follows.\\
Since the angles are smaller than $\pi$, simple closed geodesics on $\dot{S}$ avoid
the conical points and so their lengths are detected by the traces of the monodromy.
Moreover, any pair of pants decomposition of $\dot{S}$ gives rise to Fenchel-Nielsen coordinates (and so $\Ycal(\dot{S},\bm{\ell})$ is diffeomorphic to
$(\RR_+\times\RR)^{3g-3+n}$) and this quickly leads to the injectivity of $\hol\circ\Xi_{\bm{\ell}}$. The collar lemma allows to prove
properness, and one concludes that $\hol\circ\Xi_{\bm{\ell}}$ is indeed a diffeomorphism. In this case, the action
of the mapping class group is properly discontinuous and with finite stabilizers, and so very far from being ergodic.

Concerning Question \ref{q:ergodic}, Goldman has conjectured that the mapping class group of a closed surface $S$
acts ergodically on the components $\Rep(S,\PSL_2(\RR))_e$ with $e\neq 0,\,\pm\chi(S)$.
A positive answer to Question \ref{q:ergodic} would immediately imply a positive answer to Question \ref{q:zero}.
%


\section{Parabolic Higgs bundles}\label{sec:parabolic}

In this section, $S$ will denote a compact connected Riemann surface
endowed with complex structure $I$.

\subsection{Parabolic structures}

Let $E$ be a holomorphic vector bundle
on $S$, which we constantly identify with the locally-free
sheaf of its sections.

\begin{definition}
Let $E$ be a holomorphic vector bundle on $S$.
A {\it{parabolic structure}} on $E$ over $(S,P)$ is
a filtration  
$\RR\ni w\mapsto E_w \subset E(\infty \cdot P)$
of the sheaf $E(\infty\cdot P)$ of sections which are meromorphic at $P$
such that
\begin{itemize}
\item[(a)]
$E_{w}\supseteq E_{w'}$ if $w\leq w'$ {\it{(decreasing)}}
\item[(b)]
for every $w\in\RR$ there exists $\e>0$ such that
$E_{w-\e}=E_{w}$ ({\it{left-continuous}})
\item[(c)]
$E_{0}=E$ and
$E_{w+1}=E_{w}(-P)$ ({\it{normalized}}).
\end{itemize}
We will denote by $E_\bullet$ the datum of the bundle $E$ and
the given parabolic structure.
\end{definition}

\begin{notation}
Denote by $E_{p_i}$ the space of regular germs at $p_i$ of sections of $E$ 
and by $E_{p_i}(\infty\cdot p_i)$ the space of germs
at $p_i$ of sections of $E$ which are meromorphic at $p_i$.
Given a parabolic structure on $E$, the induced filtration on
$E_{p_i}(\infty\cdot p_i)$ is denoted by $w\mapsto E_{p_i,w}$.
\end{notation}

By definition, the {\it{jumps}} in the filtration at $p_i$
occur at those weights $w$ such that 
$E_{p_i,w}\supsetneq E_{p_i,w+\e}$ for all $\e>0$.
Thus, a parabolic structure is equivalent to the datum
of a weights
\[
0\leq w_1(p_i)<w_2(p_i)<\dots<w_{b_i}(p_i)< w_{b_i+1}(p_i)=1
\]
for some $b_i\in[1,\rk(E)+1]$
and a filtration
\[
E_{p_i}=E_{p_i,w_1(p_i)}
\supsetneq\dots\supsetneq E_{p_i,w_{b_i}(p_i)}\supsetneq E_{p_i,w_{b_i+1}(p_i)}=E_{p_i}(-p_i)
\]
for each $p_i\in P$.

\begin{notation}
We use the symbol $\bm{w}$ to denote the collection of
$(w_k(p_i),m_k(p_i))_{i=1}^n$, where 
$m_k(p_i)=\dim(E_{p_i,w_k(p_i)}/E_{p_i,w_{k+1}(p_i)})$ and we will say that
the parabolic bundle $E_\bullet$ is of {\it{type $\bm{w}$}}.
We will also write $\NO{\bm{w}}=(\NO{w(p_1)},\dots,\NO{w(p_n)})$, where
$\NO{w(p_i)}=\sum_{k=1}^{b_i} m_k(p_i)w_k(p_i)$, and we will say that
$\bm{w}$ is {\it{integral}} if $\NO{\bm{w}}\in\NN^n$.
\end{notation}

The {\it{(parabolic) degree}} of $E_\bullet$ is defined as
$\dis
\deg(E_\bullet):=\deg(E)+\sum_{p_i\in P}\NO{w(p_i)}
$.\\

Every holomorphic bundle $E$ can be endowed with a {\it{trivial parabolic
structure}}, by choosing $b_i=1$ and $w_1(p_i)=0$ for all $i=1,\dots,n$.
This provides an embedding on the category of holomorphic bundles on $S$
inside the category of parabolic bundles on $S$.\\

Direct sums, homomorphisms and tensor products of parabolic bundles
are defined as
\begin{align*}
(E\oplus E')_{p_i,w} & :=E_{p_i,w}\oplus E'_{p_i,w}\\
\Hom(E_\bullet,E'_\bullet) & :=\left\{
f\in\Hom(E,E')
\ \Big|
\ w_j(p_i)>w'_k(p_i)\implies f(E_{p_i,w_j(p_i)})\subseteq E'_{p_i,w'_{k+1}(p_i)}
\right\}\\
(E\otimes E')_{p_i,w''}& :=\left(\bigcup_{w+w'=w''}E_{p_i,w}\otimes E'_{p_i,w'}\right)\subset (E\otimes E')_{p_i}(\infty \cdot p_i).
\end{align*}
and we will write $\bm{w}\otimes\bm{w'}$ for the type of a parabolic bundle $E_\bullet\otimes E'_\bullet$
obtained by tensoring $E_\bullet$ of type $\bm{w}$ with $E'_\bullet$ of type $\bm{w'}$.

It is also possible to define a $\mathcal{H}om$-sheaf just by letting
$\mathcal{H}om(E_\bullet,E'_\bullet)(U):=\Hom\left(E_\bullet|_U,E'_{\bullet}|_U\right)$
with
$\mathcal{H}om(E_\bullet,E'_\bullet)_{p_i,w}=\{\text{germs at $p_i$ of morphisms $E_\bullet\rar E'_{\bullet+w}$}\}$
%
and also a dual $E_\bullet^\vee:=\mathcal{H}om(E_\bullet,\Ocal_S)$.

We will say that a homomorphism is {\it{injective}} if it is so
as a morphism of sheaves, namely if it is injective at the general point of $S$,
and {\it{properly injective}} if it is injective but not an isomorphism.
A {\it{parabolic sub-bundle}} of $E_\bullet$ is just a sub-bundle
$F\subseteq E$, endowed with the induced filtration $F_w:=F\cap E_w$;
the quotient bundle $E/F$ can be also endowed with a natural parabolic
structure by letting $(E/F)_w$
be the image of $E_w$ under the natural projection $E\rar E/F$.\\

Since $H^0(U,E_\bullet)=\Hom(\Ocal_U,E_\bullet|_U)=\Hom(\Ocal_U,E|_U)=H^0(U,E)$, sections of $E_\bullet$ are sections of $E$ and
so the same holds for higher cohomology groups.\\

In order to understand parabolic structures, it is enough to localize
the analysis and consider bundles on a disk with parabolic structure
at the origin. The typical setting is the following.

\begin{example}[Flat vector bundles on a punctured disk]\label{example:disk}
Let $N>0$ be an integer and let $\dot{\Delta}=\Delta\setminus\{p\}$ with $p=0$.
Let $\HH\rar\dot{\Delta}$ be the universal cover, defined as
$u\mapsto z=\exp(2\pi \sqrt{-1}\cdot u)$, and let $b\in\dot{\Delta}$ be a base-point.
Call $\wti{V}:=\HH\times\CC^N\rar \HH$
the trivial vector bundle and endow it with the natural connection $\wti{\nabla}$ that
can be expressed as $[\wti{\nabla}]_{\wti{\Vcal}}=d$
with respect to the canonical basis $\wti{\Vcal}=\{\ti{v}_1,\dots,\ti{v}_n\}$ of sections of $\wti{V}$
and a natural holomorphic structure $\ol{\pa}^{\wti{V}}$.\\
Given $T=\exp(-2\pi\sqrt{-1}\cdot M)\in\GL_N(\CC)$, one can lift the natural action
of $\pi_1(\dot{\Delta},b)=\langle\gamma\rangle$ on $\HH$ to an action
on $\wti{V}$ by letting $\gamma\cdot(u,v)=(u+1,T(v))$. The induced
bundle $\dot{V}:=\wti{V}/\pi_1(\dot{\Delta},b)$ 
inherits a flat connection
$\nabla$ that can be written as $[\nabla]_{\Vcal}=d$
with respect to the basis $\Vcal=\{v_1,\dots,v_N\}$ of 
flat $\ol{\pa}^V\!-$holomorphic multi-sections of $\dot{V}$.
Moreover, chosen the standard determination of $\log(z)$ on $\dot{\Delta}$,
\[
\left(
\begin{array}{c} v'_1\\ \vdots\\ v'_N\end{array}
\right):=
\exp\Big(\log(z)\cdot M\Big)
\left(
\begin{array}{c} v_1\\ \vdots\\ v_N\end{array}
\right)
\]
defines a basis $\Vcal'=\{v'_1,\dots,v'_N\}$ of univalent 
$\ol{\pa}^V\!-$holomorphic sections of $\dot{V}$
such that $[\nabla]_{\Vcal'}=d+M\frac{dz}{z}$.
Notice that $\nabla$ and so $\Res_p(\nabla)\in\End(V|_p)$ are not uniquely defined by $T$,
since $\exp(-2\pi\sqrt{-1}\,\bullet):\gl_N(\CC)\rar \GL_N(\CC)$ is not injective: in particular,
the eigenvalues of $\Res_p(\nabla)$ are only well-defined in $\CC/\ZZ$.
%
\end{example}

\subsubsection{Rank $1$}

A parabolic structure at $P$ on a line bundle $L\rar S$
is just the datum of a weight $w(p_i)\in [0,1)$ for each $p_i\in P$, so that
it makes sense to write $L_\bullet=L(\sum_i w(p_i)p_i)$.
The parabolic degree of $L_\bullet$ is simply $\deg(L_\bullet)=\deg(L)+\sum_i w(p_i)$.

\begin{notation}
If $L_\bullet=L(\sum_i w(p_i)p_i)$ is a parabolic line bundle, we denote by
$\floor{L_\bullet}$ its {\it{integral part}}, namely
the underlying line bundle $L$ with trivial parabolic structure
and by $\{L_\bullet\}:=L_\bullet \otimes \floor{L_\bullet}^\vee=\Ocal_S(\sum_i w(p_i)p_i)$ its {\it{fractional part}}.
\end{notation}

We incidentally remark that integral parabolic structures in rank $1$ are trivial.

\begin{example}[Unitary line bundles on a punctured disk]\label{example:u(1)}
Keep the notation as in Example \ref{example:disk} and let $N=1$, $v=v_1$ and $T=\exp(-2\pi\sqrt{-1}\cdot\lambda)\in \U_1$ with $\lambda\in[0,1)$.
With respect to $\Vcal'=\{v'\}$, the flat connection
\[
[\nabla]_{\Vcal'}=d+\lambda\frac{dz}{z}
\]
on $\dot{V}$ with $[\Res_p(\nabla)]_{\Vcal'}=\lambda$
has monodromy $\langle T\rangle$ and $v=z^{-\lambda}v'$
is a flat holomorphic multi-section of $\dot{V}$.
%
%
Let $V\rar\Delta$ be the extension of $\dot{V}$ 
defined by requiring that $v'$ is a generator.
We can put on $V$ a $\nabla$-invariant metric $H$
by prescribing that $\|v\|_H=c>0$ is constant, namely
\[
\|v'\|_H:=c|z|^{\lambda}
\]
and so such an invariant norm has a zero of order $\lambda\in[0,1)$ at $0$.
We can then put on $E:=V$ the same holomorphic structure
as $V$ and take $e=v'$ as a holomorphic generator of $E$, and define the filtration $E_w$ by
\[
H^0(U,E_{w}) =
\left\{s\in H^0(U,E)\ \Big|\
\begin{array}{c}
\text{on every neighbourhood $U'\Subset U$ of $p$}\\
\text{$|z|^{\e-w}\|s\|_H$ is bounded for all $\e>0$}\end{array}
\right\}
\]
so that $H^0(U,E)=H^0(U,E_w)$ for every $w\in(-1+\lambda,\lambda]$
and the jumps occur at weights in $\lambda+\ZZ$.
\end{example}

The determinant $\det(E_\bullet)$ of a parabolic bundle $E_\bullet$ of rank $N$
is the parabolic line bundle defined in the standard way as a quotient
of $E_\bullet^{\otimes N}$ by the alternating action of $\mathfrak{S}_N$.
%
If $E_\bullet$ is of type $\bm{w}$,
then
\[
\det(E_\bullet)\cong\det(E)\otimes\left(\bigotimes_{i=1}^n \Ocal_S\big(\NO{w(p_i)}p_i\big)\right)
\]
In particular, if $\bm{w}$ is integral, then $\det(E_\bullet)$ has trivial parabolic structure.

%
%
%
%
%

\subsubsection{Rank $2$ of integral type}

Among all parabolic bundles of rank $2$ we focus on 
those of integral type because of their relation with $\SL_2$-bundles.


\begin{remark}
A parabolic structure $E_\bullet$ on a bundle $E$ of rank $2$ on $(S,P)$ is of integral type $\bm{w}$ (or, briefly, just {\it{integral}}) if 
and only if the following condition holds for every $p_i\in P$:
\begin{itemize}
\item[(1)]
either $b_i=1$ ({\it{degenerate case}}) and $w_1(p_i)\in\left\{0,\frac{1}{2}\right\}$;
\item[(2)]
or $b_i=2$ ({\it{non-degenerate case}}) and $\left(0,\frac{1}{2}\right)\ni w_1(p_i)<w_2(p_i)=1-w_1(p_i)\in \left(\frac{1}{2},1\right)$.
\end{itemize}
If $E_\bullet$ is integral of rank $2$, then
$\deg(E_\bullet)=\deg(E)+\#\{p_i\in P\,|\, w_1(p_i)>0\}$.
\end{remark}

We remark that, if $E_\bullet$ is non-degenerate at $p_i$ (i.e. $b_i=2$),
then giving $E_{p_i,w_2(p_i)}$ is equivalent to
giving a line $L_i\subset E|_{p_i}$. Indeed, knowing
$E_{p_i,w_2(p_i)}$, the line
$L_i$ can be recovered as the kernel of
$E|_{p_i}\rar \left(E_{p_i}/E_{p_i,w_2(p_i)}\right)$.
Vice versa, given $L_i$, the germ $E_{p_i,w_2(p_i)}$
is the kernel of $E_{p_i}\rar E|_{p_i}/L_i$.

\begin{example}[Type of parabolic line sub-bundle]\label{example:sub-bundle}
Let $E_\bullet$ be an integral parabolic bundle of rank $2$ and let $F\subset E$ be a sub-bundle of rank $1$.
Then the jump of the induced parabolic structure on $F$ at $p_i$ occurs at $w_F(p_i)$, where
\[
\begin{array}{|c|c|c|}
\hline
& \text{$E_\bullet$ degenerate at $p_i$} & \text{$E_\bullet$ non-degenerate at $p_i$}\\
\hline
w_F(p_i) & w_1(p_i) & 
\begin{array}{cl}
w_1(p_i) & \text{if $F|_{p_i}\neq L_i$}\\
1-w_1(p_i) & \text{if $F|_{p_i}=L_i$}
\end{array}
\\
\hline
\end{array}
\]
%
%
\end{example}

The following example illustrates
how parabolic structures on holomorphic vector bundles arise from unitary representations: since the argument is local, we
will only deal with the case of a bundle over a disk.
A complete treatment can be found in \cite{mehta-seshadri:parabolic}.

\begin{example}[Rank $2$ special unitary vector bundles on a punctured disk]\label{example:disk-rank2}
Keep the notation as in Example \ref{example:disk} and let $N=2$ and $T\in \SU_2$.
Up to conjugation, we can assume that $T$ is diagonal
and that $T(v_1)=\exp(-2\pi\sqrt{-1}\cdot \lambda)v_1$ and $T(v_2)=\exp(2\pi \sqrt{-1}\cdot \lambda)v_2$
with $\lambda\in[0,1)$, and so $T$ acts on $\CC\PP^1$ as a positive rotation of angle $4\pi\lambda$
that fixes $[1:0]$.
The connection $\nabla$ on $\dot{V}$ defined as
\[
[\nabla]_{\Vcal'}=d+\left(\begin{array}{cc}\lambda & 0\\ 0 & -\lambda\end{array}\right)\frac{dz}{z}
\quad
\text{with}
\ [\Res_p(\nabla)]_{\Vcal'}=\left(\begin{array}{cc}\lambda & 0\\ 0 & -\lambda\end{array}\right)
\] 
with respect to the basis $\Vcal'=\{v'_1=z^\lambda v_1,\ v'_2=z^{-\lambda}v_2\}$ has monodromy $\langle T\rangle$.
Put on $\dot{V}$ a $\nabla$-invariant metric $H$ by prescribing that
$v_1,v_2$ are orthogonal with $\|v_i\|_H=c_i>0$ and
define the univalent sections $e_1,e_2$ of $\dot{V}$
as
\[
e_1=v'_1,
\qquad
e_2=
\begin{cases}
v'_2 & \text{if $\lambda=0$}\\
z\cdot v'_2 & \text{if $\lambda>0$}
\end{cases}
\]
and extend $\dot{V}$ to a vector bundle $V\rar\Delta$ with generators $e_1,\ e_2$.
The holomorphic vector bundle $E:=V$ endowed with the same holomorphic structure as $V$
has a pointwise orthogonal
basis $\Ecal=\{e_1,e_2\}$ of holomorphic sections, that satisfy
\[
\|e_1\|_H=c_1|z|^{\lambda},
\qquad
\|e_2\|_H=
\begin{cases}
c_2 & \text{if $\lambda=0$}\\
c_2|z|^{1-\lambda} & \text{if $\lambda>0$}
\end{cases}
\]
If $\lambda=0$, then $\|e_i\|_H=c_i$ and 
we have a degenerate parabolic structure with
$w_1=0$. Similarly, if $\lambda=\frac{1}{2}$, then
$\|e_i\|_H=c_i |z|^{\frac{1}{2}}$ and $w_1=\frac{1}{2}$.

Assume now that $\lambda>0$ but $\lambda\neq\frac{1}{2}$.\\
If $\lambda\in \left(0,\frac{1}{2}\right)$, then we let $w_1=\lambda<w_2=1-\lambda$
and $L=\CC e_2\times\{p\}\subset E|_{p}$.
Since $|z|^\lambda>|z|^{1-\lambda}$, a section $s=f_1(z)e_1+f_2(z)e_2$ satisfies
\[
\ord_{p}\|s\|_H=
\ord_{p}\left(|f_1 z^\lambda|+|f_2 z^{1-\lambda}|\right)=
\begin{cases}
w_1=\lambda & \text{if $f_1(p)\neq 0$, i.e. if $s(p)\notin L$}\\
w_2=1-\lambda & \text{if $f_1(p)=0$, i.e. if $s(p)\in L$.}
\end{cases}
\]
If $\lambda\in \left(\frac{1}{2},1\right)$, then we let $w_1=1-\lambda<w_2=\lambda$
and $L=\CC e_1\times\{p\}\subset E_p$.
In both cases,
$E_{w}$ is defined by 
\[
H^0(U,E_{w}) =
\{s\in H^0(U,E)\,|\,\|s\|_H\cdot |z|^{\e-w}\ \text{bounded near $p$ for all $\e>0$}\}
\]
and in particular, $H^0(U,E_{w_2})=\{s\in H^0(U,E)\,|\,s(p)\in L\}$ and
\[
H^0(U,E_w)=
\begin{cases}
H^0(U,E) & \text{for $w\in[0,w_1]$}\\
H^0(U,E_{w_2}) & \text{for $w\in (w_1,w_2]$}\\
H^0(U,E(-p)) & \text{for $w\in(w_2,1)$}.
\end{cases}
\]
Observe that such parabolic structure is integral.
\end{example}

\subsection{Higgs bundles}

\begin{definition}
A {\it{parabolic Higgs bundle}} on $(S,P)$ is a couple 
$(E_\bullet,\Phi)$, where $E_\bullet$ is a holomorphic parabolic vector bundle
of rank $N$ and $\Phi\in H^0(S,K(P)\otimes\End(E_\bullet))$. The {\it{residue}}
of $\Phi$ at $p_i$ is the induced endomorphism $\Res_{p_i}(\Phi)\in \End(E_\bullet|_{p_i})$
of the filtered vector space $E_\bullet|_{p_i}$.
\end{definition}

Here is the motivating example in rank $N=1$ on the punctured disk.

\begin{example}[Line bundles on a punctured disk]\label{example:rank1-higgs}
Keep the notation as in Example \ref{example:disk} and let $N=1$ and $T=\exp[-2\pi \sqrt{-1}(\lambda+i\nu)]\in \GL_1(\CC)=\CC^*$ with
$\lambda\in [0,1)$ and $\nu\in\RR$. We can assume that $\nu\neq 0$, as
the case $\nu=0$ has already been discussed in Example \ref{example:u(1)}.
Such a monodromy $T$ is induced by
a flat connection $\nabla$ that can be written as
\[
[\nabla]_{\Vcal'}=d+(\lambda+i\nu)\frac{dz}{z}
\quad
\text{with}
\ [\Res_p(\nabla)]_{\Vcal'}=\lambda+i\nu
\]
with respect to $\Vcal'=\{v'=v'_1\}$.
The metric $H$ on $\dot{V}$ defined by
\[
\|v'\|_H:=c|z|^\lambda
\]
is harmonic with respect to $\nabla$, 
since $i\pa\ol{\pa}\log\|v'\|^2_H=0$.
Thus,
\begin{align*}
[\nabla]_{\Vcal'} =[\nabla^H]_{\Vcal'} +\Phi +\ol{\Phi}= \left(d+\lambda\frac{dz}{z}+i\frac{\nu}{2}\frac{dz}{z}+i\frac{\nu}{2}\frac{d\ol{z}}{\ol{z}}\right) +
\left(i\frac{\nu}{2}\frac{dz}{z}\right) +\left(-i\frac{\nu}{2}\frac{d\ol{z}}{\ol{z}}\right)
\end{align*}
where $\nabla^H$ is a connection on $\dot{V}$ compatible with
the metric $H$ and $\Res_p(\Phi)=i\nu/2$.
Extend $\dot{V}$ to the bundle $V=\CC v'\times\Delta\rar\Delta$ and put on the complex
line bundle $E:=V$ the holomorphic structure
given by $\ol{\pa}^E:=\ol{\pa}^V-\ol{\Phi}$, 
so that $\nabla^H$ is a Chern connection on $(E,\ol{\pa}^E,H)$.
Holomorphic sections of $E$ are generated then by
$e=\exp(-i\nu\log|z|)v'$.
Since $\|e\|_H=\|v'\|_H=c|z|^{\lambda}$, 
the filtration $E_\bullet$ is as in Example \ref{example:u(1)},  and
so that the jumps occur at $\lambda+\ZZ$.
\end{example}


The following computation is borrowed from \cite{simpson:harmonic}.

\begin{example}[Flat $\SL_2$-vector bundles on a punctured disk]\label{example:log}
Keep the notation as in Example \ref{example:disk} and let $N=2$
and $T\in \SL_2(\CC)$. If $T$ is diagonalizable, then the bundle $\dot{E}$ splits and we are reduced to the rank $1$ case
of Example \ref{example:rank1-higgs}. Thus, up to conjugation, we can assume that
\[
\left[T\right]_{\Vcal'}=\frac{1}{2}\left(
\begin{array}{cc}
1 & -1\\
0 & 1
\end{array}
\right),
\qquad
\left[\nabla\right]_{\Vcal'}
=d+
\frac{1}{2}\left(
\begin{array}{cc}
0 &1\\
0 & 0
\end{array}
\right)\frac{dz}{z}
\]
with respect to $\Vcal'=\{v'_1,v'_2\}$ and so
$[\Res_p(\nabla)]_{\Vcal'}=\frac{1}{2}\left(
\begin{array}{cc}
0 &1\\
0 & 0
\end{array}
\right)$.
Put on $\dot{V}$ the metric $H$ defined by
\[
[H]_{\Vcal'}=
\left(
\begin{array}{cc}
2|\log|z||^{-1} & -1\\
-1 & |\log|z||
\end{array}
\right)
\]
which is harmonic for $\nabla$ and
extend $\dot{V}$ to a complex bundle $V\rar S$ with basis $\Vcal'$.
Thus, $\nabla=\nabla^H+\Phi+\ol{\Phi}^H$, where 
\[
\left[\Phi\right]_{\Vcal'}=
\frac{1}{2}
\left(
\begin{array}{cc}
-|\log|z||^{-1} & -1\\
-|\log|z||^{-2} & |\log|z||^{-1}
\end{array}
\right)\frac{dz}{z},
\quad
\left[\ol{\Phi}^H\right]_{\Vcal'}=
\frac{1}{2}
\left(
\begin{array}{cc}
0 & 0\\
|\log|z||^{-2} & 0
\end{array}
\right)\frac{d\ol{z}}{\ol{z}}.
\]
and $\nabla^H$ is compatible with $H$.
Let now $\dot{E}:=\dot{V}$ as a complex vector bundle and notice that 
$\Ecal=\{e_1:=v'_1+\frac{v'_2}{|\log|z||},\ e_2:=v'_2\}$
is a set of generators for $\dot{E}$, which are
holomorphic with
respect to the operator $\ol{\pa}^E:=\ol{\pa}^V-\ol{\Phi}^H$.
Thus, we can extend $\dot{E}$ to a $E\rar\Delta$
by requiring that $e_1,e_2$ are generators.
A quick calculation gives
\[
[\Phi]_{\Ecal}=\frac{1}{2}\left(\begin{array}{cc}
0 & 1\\
0 & 0
\end{array}\right)\frac{dz}{z},
\quad
[\ol{\Phi}^H]_{\Ecal}=\frac{1}{2}\left(\begin{array}{cc}
0 & 0\\
|\log|z||^{-2} & 0
\end{array}\right)\frac{d\ol{z}}{\ol{z}}
\]
and so $\Phi\in H^0(\Delta,K(p)\otimes\End_0(\dot{E}))$
is a traceless Higgs field with nonzero nilpotent residue at $p$
and $\ol{\Phi}^H$ is its $H$-adjoint. 
From
\[
\frac{c}{|\!\log|z||^{\frac{1}{2}}}\leq \|e_j\|_H\leq c'|\!\log|z||^{\frac{1}{2}}
\]
it follows that $w_1=0$, the parabolic structure is integral degenerate and the jumps occur at $\ZZ$.

A similar computation shows that, for $T'=-T$, the norm
satisfies
\[
\frac{c|z|^{\frac{1}{2}}}{|\!\log|z||^{\frac{1}{2}}}\leq
\|e_j\|_H\leq c' |\!\log|z||^{\frac{1}{2}}|z|^{\frac{1}{2}}
\]
 and so
we would still obtain an integral degenerate parabolic structure with $w_1=\frac{1}{2}$ and
a traceless Higgs field with nonzero nilpotent residue, 
but the jumps would occur at $\frac{1}{2}+\ZZ$.
\end{example}

\medskip

A morphism of parabolic Higgs bundles $f:(E_\bullet,\Phi)\rar (E'_\bullet,\Phi')$
is a map $f:E_\bullet\rar E'_\bullet$ of parabolic bundle that makes the following
diagram
\[
\xymatrix{
E \ar[rr]^{\Phi} \ar[d]^f && E\otimes K(P)\ar[d]^{f\otimes 1}\\
E' \ar[rr] ^{\Phi'} && E'\otimes K(P)
}
\]
commutative. We will say that $f$ is injective (resp. properly injective)
if $f:E\rar E'$ is.
A parabolic Higgs sub-bundle of $(E_\bullet,\Phi)$ is a sub-bundle 
$F_\bullet\subseteq E_\bullet$ such that $\Phi(F)\subseteq F\otimes K(P)$;
the map $\Psi:(E/F)_\bullet\rar (E/F)_\bullet\otimes K(P)$ induced by $\Phi$
makes $\left((E/F)_\bullet,\Psi\right)$ into a parabolic Higgs quotient bundle.\\


\subsection{Stability}

The {\it{slope}} of the parabolic bundle $E_\bullet$ on $(S,P)$ is
defined as $\mu(E_\bullet):=\frac{\deg(E_\bullet)}{\rk(E)}$.

\begin{definition}
A parabolic Higgs bundle $(E_\bullet,\Phi)$ on $(S,P)$ is {\it{stable}}
(resp. {\it{semi-stable}}) if $\mu(F_\bullet)<\mu(E_\bullet)$
(resp. $\mu(F_\bullet)\leq \mu(E_\bullet)$) for every properly
injective $(F_\bullet,\Psi)\rar (E_\bullet,\Phi)$. A direct sum of stable
parabolic Higgs bundles with the same $\mu$ is said {\it{polystable}}.
\end{definition}

\begin{remark}\label{rmk:simple}
It is well-known that stable bundles are {\it{simple}}, i.e. their endomorphisms
are multiples of the identity, and so the group of their automorphism is 
$\CC^*$. The same argument works for parabolic Higgs bundles.
Indeed, if $f:(E_\bullet,\Phi)\rar (E_\bullet,\Phi)$ is a non-zero homomorphism, then 
$\mu(E_\bullet)\leq \mu(\IM(f))\leq \mu(E_\bullet)$ by semistability of $(E_\bullet,\Phi)$.
This forces $\IM(f)=E_\bullet$ because $(E_\bullet,\Phi)$ is stable.
Now pick a point $q\in S$ and let $\lambda\in\CC$ be an eigenvalue of $f_q:E|_q\rar E|_q$.
The endomorphism $(f-\lambda \cdot \id)\in \End(E_\bullet,\Phi)$ is not surjective and so it vanishes
by the above argument. It follows that $f=\lambda \cdot \id$.
\end{remark}

Semi-stable parabolic Higgs bundles have the Jordan-H\"older property:
if $(E_\bullet,\Phi)$ is semi-stable, then there exists a filtration
\[
\{0\}=E^0_\bullet\subsetneq E^1_\bullet\subsetneq E^2_\bullet\subsetneq\dots\subsetneq 
E_\bullet
\]
by parabolic sub-Higgs-bundles such that 
the Higgs bundle structure $\Gr^s(E_\bullet,\Phi)$
induced on the quotient $E^s_\bullet/E^{s-1}_\bullet$ is stable
and with slope $\mu(\Gr^s(E_\bullet))=\mu(E_\bullet)$.
It can be checked that, though the filtration is not canonical, the associated graded object
\[
\Gr(E_\bullet,\Phi)=\bigoplus_{s} \Gr^s(E_\bullet,\Phi)
\]
is. As for vector bundles, two parabolic Higgs bundles $(E_\bullet,\Phi)$, $(E'_\bullet,\Phi')$ are called {\it{$S$-equivalent}} if $\Gr(E_\bullet,\Phi)\cong \Gr(E'_\bullet,\Phi')$.
Thus, every semistable object is $S$-equivalent to a unique polystable one, up to isomorphism.

\subsection{Moduli spaces of parabolic $\SL_n$-Higgs bundles}

Fix a type $\bm{w}$, an $N$-uple $\bCoa$ of conjugacy classes in $\psl_N(\CC)$ and a base-point $b\in\dot{S}$.
Fix also a holomorphic parabolic line bundle $\Det_\bullet$ of type $\NO{\bm{w}}$.

\begin{definition}
A {\it{parabolic Higgs bundle of rank $N$ with determinant $\Det_\bullet$}} on $(S,P)$ 
{\it{of type $\bm{w}$}}
is a triple
$(E_\bullet,\eta,\Phi)$, where $E_\bullet$ is a holomorphic parabolic vector bundle
of rank $N$ and type $\bm{w}$, endowed with an isomorphism
$\eta:\det(E_\bullet)\arr{\sim}\Det_\bullet$
and a {\it{Higgs field}} $\Phi\in H^0(S,K(P)\otimes\End_0(E_\bullet))$.
An isomorphism $(E_\bullet,\eta,\Phi)\rar (E'_\bullet,\eta',\Phi')$
of parabolic Higgs bundles of rank $N$ with determinant $\Det_\bullet$
is a map $f:E_\bullet\rar E'_\bullet$ which is an isomorphism of parabolic Higgs bundles
and such that $\eta=\eta'\circ\det(f)$.
\end{definition}

By Remark \ref{rmk:simple}, 
an automorphism $f$ of $(E_\bullet,\eta,\Phi)$ must satisfy $\det(f_x)=1$ at all $x\in S$.
Hence, if $(E_\bullet,\Phi)$ is simple, then $\Aut(E_\bullet,\eta,\Phi)=\mu_N\cdot\id$,
where $\mu_N\subset\CC^*$ is the cyclic subgroup of $N$-th roots of unity.\\

Denote by $\Higgs^{ss}_b(S,N,\bm{w},\Det_\bullet,\bCoa)$ the set of isomorphism classes of
quadruples $(E_\bullet,\eta,\Phi,\tau)$ such that
\begin{itemize}
\item
$(E_\bullet,\eta,\Phi)$ is a semistable parabolic Higgs bundle on $(S,P)$ of rank $N$, type $\bm{w}$ and with determinant $\Det_\bullet$
\item
$\Res_{p_i}(\Phi)\in\coa_i$ for all $i=1,\dots,n$
\item
$\tau:E|_b\arr{\sim}\CC^N$ is a framing at $b$.
\end{itemize}
We denote by $\Higgs^{s}_b\subseteq\Higgs^{ps}_b\subseteq\Higgs^{ss}_b$ the stable and polystable loci.

The following two results are due to Simpson \cite{simpson:harmonic}, Konno \cite{konno:construction} and Yokogawa
\cite{yokogawa:compactification}.

\begin{theorem}[Moduli space of framed semi-stable parabolic Higgs bundles]
The space $\Higgs_b^{ss}(S,N,\bm{w},\Det_\bullet,\bCoa)$ is a normal quasi-projective variety
and a fine moduli space of $b$-framed semi-stable Higgs bundles
on $(S,P)$ of rank $N$, type $\bm{w}$ and with determinant $\Det_\bullet$.
Moreover, the stable locus $\Higgs_b^s(S,N,\bm{w},\Det_\bullet,\bCoa)$ is smooth.
\end{theorem}

The group $\GL_N(\CC)$ acts by post-composition on the $b$-framing, and so it acts on $\Higgs_b^{ss}(S,N,\bm{w},\Det_\bullet,\bCoa)$:
we denote by $\RRHiggs(S,N,\bm{w},\Det_\bullet;\bCoa)$ the set-theoretic quotient
and by $\Higgs(S,N,\bm{w},\Det_\bullet,\bCoa)$ its Hausdorffization.

\begin{theorem}[Moduli space of stable parabolic Higgs bundles]
The Hausdorff quotient $\Higgs(S,N,\bm{w},\Det_\bullet,\bCoa)$ is a normal quasi-projective variety, whose points are in bijection with $S$-equivalence classes of
semi-stable Higgs bundles on $(S,P)$ of rank $N$, type $\bm{w}$ and with determinant $\Det_\bullet$.
The open locus 
$\Higgs^s(S,N,\bm{w},\Det_\bullet,\bCoa)$
is an orbifold and a fine moduli space of stable objects. 
\end{theorem}

Let $F_\bullet$ be a line bundle on $S$ with parabolic structure of type $\bm{w_F}$ (which is trivial at $b$) 
and fix a trivialization of $F_\bullet$ at the basepoint $b\in\dot{S}$.
Then $E_\bullet\mapsto F_\bullet\otimes E_\bullet$ induces a $\GL_N(\CC)$-equivariant isomorphism
\[
\Higgs^{ss}_b(S,N,\bm{w},\Det_\bullet,\bCoa)\arr{\sim}
\Higgs^{ss}_b(S,N,\bm{w_F}\otimes\bm{w},F_\bullet^{\otimes N}\otimes\Det_\bullet,\bCoa)
\]
that preserves polystable and stable locus. Thus, it induces an isomorphism
\[
\Higgs(S,N,\bm{w},\Det_\bullet,\bCoa)\cong
\Higgs(S,N,\bm{w_F}\otimes \bm{w},F_\bullet^{\otimes N}\otimes\Det_\bullet,\bCoa)
\]
that preserves the stable locus.

\begin{remark}
Let $L_\bullet=L(\sum_{i=1}^n w(p_i)p_i)$ 
be a parabolic line bundle on $(S,P)$ and let $N\geq 2$.
Chosen $p_0\in S$,
there exist an integer $r$ and
a line bundle $Q$ such that $L\cong Q^{\otimes N}\otimes \Ocal_S(r\cdot p_0)$ with $0\leq r<N-1$.
Thus, $L_\bullet\cong Q(\frac{r}{N}p_0+\sum_{i=1}^n\frac{w(p_i)}{N}p_i)^{\otimes N}$.
Hence, if $n>0$, then we can choose $p_0=p_1$ for instance, and so
$L_\bullet$ admits an $N$-root which is a line bundle on $S$ with parabolic structure at $P$.
If $n=0$, then $L_\bullet$ admits an $N$-th root which is a line bundle (possibly) with parabolic structure
at $p_0$.
\end{remark}

By the above remark, we can choose an $N$-th root
$F_\bullet$ of $\Det^\vee_\bullet$ and so
we have established an isomorphism
$\Higgs^s(S,N,\bm{w},\Det_\bullet,\bCoa)\rar\Higgs^s(S,N,\bm{w_F}\otimes \bm{w},\Ocal_S,\bCoa)$.

For rank $2$ integral parabolic structures the remark specializes to the following.

\begin{corollary}[Odd and even rank $2$ integral parabolic structures]\label{cor:det}
Let $N=2$ and fix an integral parabolic type $\bm{w}$ and a line bundle $\Det$ with trivial parabolic structure.
Also fix an auxiliary point $p_0\in\dot{S}$ different from $b$ and let $\bm{w_0}$ be the parabolic type of $\Ocal_S(\frac{1}{2}p_0)$.\\
Then the moduli space $\Higgs^{s}(S,2,\bm{w},\Det,\bCoa)$ is isomorphic to either of the following:
\begin{itemize}
\item[(1)]
$\Higgs^{s}(S,2,\bm{w}_0\otimes\bm{w},\Ocal_S,\bCoa)$, if $\deg(\Det)$ is odd;
\item[(2)]
$\Higgs^{s}(S,2,\bm{w},\Ocal_S,\bCoa)$, if $\deg(\Det)$ is even.
\end{itemize}
\end{corollary}

\subsection{Involution and fixed locus}

We now restrict to the case of rank $N=2$,
and we remark that
every element $X\in\psl_2(\CC)$ is in the same $\Ad_{\SL_2(\CC)}$-orbit as $-X$.


Fix integral $\bm{w}$,
an $n$-uple of conjugacy classes $\bCoa$ in $\psl_2(\CC)$ and
a line bundle $\Det$ with trivial parabolic structure and let
$d_0=\deg(\Det)$.

Following Hitchin, consider the involution
$\sigma_b:\Higgs^{ss}_b(S,2,\bm{w},\Det,\bCoa)\rar\Higgs^{ss}_b(S,2,\bm{w},\Det,\bCoa)$
defined as $\sigma(E_\bullet,\eta,\Phi,\tau):=(E_\bullet,\eta,-\Phi,\tau)$,
and let $\sigma$ be the induced map on $\Higgs(S,2,\bm{w},\Det,\bCoa)$.

\begin{lemma}[$\sigma$-fixed locus]\label{lemma:real}
Let $(E_\bullet,\eta,\Phi)$ be a polystable parabolic Higgs bundle of rank $2$, type $\bm{w}$, determinant $\Det$ and with residues at $p_i$ in $\coa_i$.
The point $[E_\bullet,\eta,\Phi]$ is fixed under $\sigma$ if and only if 
there exists an isomorphism $\iota:(E_\bullet,\eta,\Phi)\arr{\sim} (E_\bullet,\eta,-\Phi)$, which happens
if and only if one of the following conditions is satisfied:
\begin{itemize}
\item[(a)]
$\Phi=0$ and $E_\bullet$ is polystable;
\item[(b)]
$E_\bullet\cong (L^\vee_\bullet\otimes\Det)\oplus L_\bullet$ 
with $\deg(\Det)\leq 2\deg(L_\bullet)$
and
\[
0\neq\Phi=\left(\begin{array}{cc}
0 & \phi\\
\psi & 0
\end{array}\right),
\qquad
\iota=\pm\left(\begin{array}{cc}
i & 0\\
0 & -i
\end{array}\right)
\]
with respect to this decomposition, where
\begin{align*}
0\neq & \phi \in \Hom(L_\bullet,L^\vee_\bullet\otimes\Det)\otimes K(P)\\
&\psi \in \Hom(L^\vee_\bullet\otimes\Det,L_\bullet)\otimes K(P).
\end{align*}
If $\deg(\Det)<2\deg(L_\bullet)$, then $(E_\bullet,\Phi)$ is necessarily stable.
If $\deg(\Det)=2\deg(L_\bullet)$, then 
$(E_\bullet,\Phi)$ is polystable if and only if $\psi\neq 0$ too.
\end{itemize}
Furthermore, if $(E_\bullet,\eta,\Phi)$ is stable with $\Phi\neq 0$ and $[E_\bullet,\eta,\Phi]$ is fixed by $\sigma$, then
\begin{itemize}
\item[(b1)]
the isomorphism $\iota$ is unique up to $\{\pm 1\}$;
\item[(b2)]
if $\Det\not\cong L^{\otimes 2}_\bullet$, then
the decomposition $E_\bullet\cong (L^\vee_\bullet\otimes\Det)\oplus L_\bullet$
is unique;
\item[(b3)]
if $\Det\cong L^{\otimes 2}_\bullet$ and so $E_\bullet\cong L_\bullet\oplus L_\bullet$,
then $\phi,\psi$ are not proportional.
\end{itemize}
Finally, if $(E_\bullet,\eta,\Phi)$ is strictly polystable with $\Phi\neq 0$
and $[E_\bullet,\eta,\Phi]$ is fixed by $\sigma$, then 
$(E_\bullet,\Phi)\cong \left(L_\bullet\oplus L_\bullet,\left(\begin{array}{cc}
\phi & 0\\
0 & -\phi
\end{array}\right)\right)$ with $0\neq \phi\in H^0(S,K(P))$ and
$\iota=\pm\left(\begin{array}{cc}
0 & -1\\
1 & 0
\end{array}\right)$.
\end{lemma}
\begin{proof}
The argument is essentially the same as in \cite{hitchin:self-duality}, Sec.~10.\\
In case (a), the Higgs bundle $(E_\bullet,0)$ is polystable. So we assume $\Phi\neq 0$
and we want to show that (b) holds.
%

Since $\iota\circ \Phi\circ \iota^{-1}=-\Phi$, it is easy to check that $\Phi$ must vanish at all points $Q$ where $\iota$ is a unipotent automorphism. Moreover, a quick computation shows that $\iota$ must have eigenvalues $\pm i$
at $\dot{S}\setminus Q$. Since $\Phi\neq 0$ except at a finite number of points, $\iota$ has everywhere eigenvalues $\pm i$ with eigenbundles $L_\bullet$ and $L^\vee_\bullet\otimes\Det$, and 
so $\iota=\pm \left(\begin{array}{cc}
i & 0\\
0 & -i
\end{array}\right)$ and
$\Phi=\left(\begin{array}{cc}
0 & \phi\\
\psi & 0
\end{array}\right)$.\\
%
Suppose first $\deg(\Det)<2\deg(L_\bullet)$.\\
Then $L_\bullet$ cannot be preserved by $\Phi$ by semi-stability,
and so $\phi\neq 0$.
Moreover, $\deg(\Det)/2>\deg(L^\vee_\bullet\otimes\Det)$ implies
a line sub-bundle $L'_\bullet\subset E_\bullet$ with $\deg(L'_\bullet)>\mu(E_\bullet)=\deg(\Det)/2$
must necessarily be $L'_\bullet=L_\bullet$. It follows that
the projection $L'_\bullet\rar L^\vee_\bullet\otimes\Det$ necessarily vanishes and so
$(E_\bullet,\Phi)$ is stable.

Suppose now $\deg(\Det)=2\deg(L_\bullet)$.\\
Then $\phi=0$ would imply $\psi=0$ by polystability (and vice versa): hence, we must have $\phi,\psi\neq 0$.

About the second part of the statement, given $\iota,\iota':(E_\bullet,\eta,\Phi)\rar (E_\bullet,\eta,-\Phi)$ isomorphisms of stable Higgs bundles with determinant $\Det$,
the composition $(\iota^{-1}\circ \iota')\in\Aut(E_\bullet,\eta,\Phi)$ and so $\iota^{-1}\circ \iota'=\pm\id$ because
$(E_\bullet,\eta,\Phi)$ is simple. Property (b2) easily follows from (b1) and (b).
As for (b3), a destabilizing sub-bundle (necessarily isomorphic to $L_\bullet$)
exists if and only if $0\neq \phi,\psi\in H^0(S,K(P))$ are proportional.

Concerning the final claim, 
it is enough to observe that
$\phi$ vanishes only on finitely many points of $S$ and
the involution $\iota$ must exchange the two eigenspaces of $\Phi/\phi$ away from those finitely many points.
\end{proof}


Denote by $\Higgs^{ps}_b(S,2,\bm{w},\Det,\bCoa)(\RR)$ the set of
$(E_\bullet,\eta,\Phi,\tau,\{\pm\iota\})$ such that
$(E_\bullet,\eta,\Phi,\tau)$ is a polystable parabolic Higgs bundle of rank $2$
of type $\bm{w}$ with determinant $\Det$ and residues in $\bCoa$,
and $\iota:(E_\bullet,\eta,\Phi)\rar (E_\bullet,\eta,-\Phi)$ is an isomorphism.
We denote by $\Higgs^s_b(S,2,\bm{w},\Det,\bCoa)(\RR)$ the locus of stable
objects and by $\Higgs^s(S,2,\bm{w},\Det,\bCoa)(\RR)$ its quotient by $\SL_2(\CC)$.

By the above lemma, $\Higgs^s(S,2,\bm{w},\Det,\bCoa)(\RR)$ can be identified to
the locus in $\Higgs^s(S,2,\bm{w},\Det,\bCoa)$ fixed by $\sigma$, and so we will
denote a point just by $(E_\bullet,\eta,\Phi)$.

If $\ol{\bCoa}\supseteq \bm{\mathfrak{0}}=\{0\}^n$, then $\Higgs^s(S,2,\bm{w},\Det,\ol{\bCoa})(\RR)$ contains
the locus of $\Phi=0$, namely
\[
\Bun^s(S,2,\bm{w},\Det) :=
\big\{\text{$(E_\bullet,\eta)$ stable parabolic rank $2$ bundle of type $\bm{w}$
with $\eta:\det(E_\bullet)\arr{\sim}\Det$}\big\}
\]

\begin{definition}
We say that the couple $(\bm{w},\bCoa)$ is {\it{compatible (with the $\sigma$-involution)}}
if
\begin{itemize}
\item
$0<w_1(p_i)<1/2\ \implies\ \bCoa_i=\{0\}$
\item
$w_1(p_i)=0;\ 1/2\ \implies\ \det(\bCoa_i)\geq 0$
\item
$s_0+\chi(\dot{S})<0$.
\end{itemize}
\end{definition}

\begin{notation}
Given $\bm{w}$, denote by $J\even_{deg}$ (resp. $J\odd_{deg}$)
the set of indices $j\in\{1,2,\dots,n\}$
for which $w_1(p_j)=0$ (resp. $w_1(p_j)=\frac{1}{2}$)
and let $J_{deg}=J\even_{deg}\cup J\odd_{deg}$.\\
Given $\bCoa$, we define
$J_{0} =\{j\in J_{deg}\,|\,\bcoa_j=\{0\}\}$, $J_{nil} =\{j\in J_{deg}\,|\,\text{$\bcoa_j$ nilpotent}\}$
and $J_{inv} =\{j\in J_{deg}\,|\,\det(\bcoa_j)\neq 0\}$.\\
Finally, call $s=\# J_{deg}$, $s_{0}=\# J_0$ and $s_{inv}=\# J_{inv}$,
and also $s\even=\# J\even_{deg}$ and $s\odd=\# J\odd_{deg}$.
\end{notation}


\begin{definition}
Given $\Det$ and compatible $(\bm{w},\bCoa)$, we say that
the couple $(d,\bm{a})\in\ZZ\times\{0,1\}^n$
is {\it{admissible}} if
\begin{itemize}
\item[(a)]
$a_j=0$ for all $j\in J_{deg}$;
\item[(b)]
$e(d,\bm{a},\bm{w}):=2d-d_0+2\sum_{i=1}^n \left(a_i+(-1)^{a_i}w_1(p_i)\right)\geq 0$
\item[(c)]
$2d\leq d_0-\chi(S)-\NO{\bm{a}}+s\even-s_0$.
\end{itemize}
\end{definition}

For every admissible $(d,\bm{a})$,
we define the following locus in $\Higgs^s(S,2,\bm{w},\Det,\bCoa)(\RR)$
\[
\Higgs^s(S,2,\bm{w},\Det,\bCoa)(\RR)_{d,\bm{a}} : =
\left\{
\begin{array}{l}
\text{$E_\bullet\cong(L^\vee_\bullet\otimes\Det)\oplus L_\bullet$
with $\Phi=\left(\begin{array}{cc} 0 & \phi\\ \psi & 0\end{array}\right)$
and $\Res_{p_i}(\Phi)\in\bcoa_i$}\\
0\neq \phi\in H^0(S,\Det K L_\bullet^{-2}(P)), \quad
\psi\in H^0(S,\Det^\vee K L_\bullet^2(P))\\
\deg(L)=d,
\quad w_L(p_i)=\begin{cases} w_1(p_i) & \text{if $a_i=0$}\\ 1-w_1(p_i) &\text{if $a_i=1$}\end{cases}
\end{array}
\right\}
\]
Thus, $e(d,\bm{a},\bm{w})$ can be rewritten as $e(d,\bm{a},\bm{w})=2d-d_0+2\sum_{i=1}^n w_L(p_i)$
and the admissibility constraints (b-c) read
\[
-2\sum_{i=1}^n w_L(p_i)\leq 2d-d_0\leq -\chi(S)-\NO{\bm{a}}+s\even-s_0.
\]
Moreover, the condition $e(d,\bm{a},\bm{w})\geq 0$ is equivalent to $\deg(L_\bullet)\geq \deg(\Det\otimes L_\bullet^\vee)$;
thus, in view of Lemma \ref{lemma:real},
it is understood that we also require $\psi\neq 0$, if $e(d,\bm{a},\bm{w})=0$.

We can rephrase our analysis as follows.

\begin{proposition}[Partition of the $\sigma$-fixed locus]\label{prop:real}
The space $\Higgs^s(S,2,\bm{w},\Det,\ol{\bCoa})$ can be decomposed into
the disjoint union of the following loci
\[
\begin{array}{ll}
\Bun^s(S,2,\bm{w},\Det) & \text{if $\bm{\mathfrak{0}}\in\ol{\bCoa}$}\\
\Higgs^s(S,2,\bm{w},\Det,\ol{\bCoa})(\RR)_{d,\bm{a}} & 
\text{for admissible $(d,\bm{a})$.}
\end{array}
\]
\end{proposition}


\subsection{Topology of the $\sigma$-fixed locus}\label{sec:topology}

Let $\bcoa_i$ be conjugacy classes and $\ol{\bcoa}_i$ be their closures in $\psl_2(\CC)$.
Throughout this section, we will assume that $(\bm{w},\bCoa)$ are
compatible, that $(d,\bm{a})$ is admissible and that
$e(d,\bm{a},\bm{w})>0$.

We begin with some simple observations.

\begin{lemma}
Let $(E_\bullet,\eta,\Phi,\{\pm\iota\})\in \Higgs^s(S,2,\bm{w},\Det,\ol{\bCoa})(\RR)_{d,\bm{a}}$ and
$E_\bullet=(L_\bullet^\vee\otimes\Det)\oplus L_\bullet$.
Then
\begin{itemize}
\item[(i)]
for $i\notin J_{deg}$ ($\bm{w}$ non-degenerate at $p_i$), $L^{-2}_\bullet(P+\bm{a}P)$ has parabolic twist
$1+a_i-2w_L(p_i)\in [0,1)$ and $L^2_\bullet(-\bm{a}P)$ has twist $2w_L(p_i)-a_i\in [0,1)$
at $p_i$;
\item[(ii)]
for $j\in J_{deg}$ ($\bm{w}$ degenerate at $p_j$), $L^2$ and $L^{-2}$ have trivial parabolic structure
at $p_j$.
\end{itemize}
Thus, $\deg\floor{\Det L_\bullet^{-2}(P)}=d_0-2d-\NO{\bm{a}}+s\even$ and
$\deg\floor{\Det^\vee L_\bullet^2}=2d-d_0+\NO{\bm{a}}+s\odd \geq 1-n$.
As a consequence,
\[
\begin{array}{rcl}
m:=& \deg\floor{\Det L^{-2}_\bullet K(P)}&=d_0+2g-2-2d-\NO{\bm{a}}+s\even\geq 0\\
& \deg\floor{\Det^\vee L^{2}_\bullet K(P)}&=-d_0+2g-2+n+2d+\NO{\bm{a}}+s\odd\geq 0\\
m':= & h^0(S,\Det^\vee L^2_\bullet K(P)) &=-d_0+g-1+n+2d+\NO{\bm{a}}+s\odd\geq 0
\end{array}
\]
since $\deg(\floor{\Det^\vee L^2_\bullet(P)})>0$. Thus, $m+m'=3g-3+n+s$.
\end{lemma}
\begin{proof}
Parts (i) and (ii) are elementary computations.
The bound for $\deg\floor{\Det^\vee L_\bullet^2}$ follows by observing that
$0<e(d,\bm{a},\bm{w})=2d-d_0+2\sum_{i=1}^n w_L(p_i)\leq 2d-d_0+\NO{\bm{a}}+n$
because $w_L(p_i)\leq a_i+1$. Thus, $m'$ can be calculated by Riemann-Roch.
\end{proof}

Since we are assuming $e(d,\bm{a},\bm{w})>0$,
a point of $\Higgs^s(S,2,\bm{w},\Det,\ol{\bCoa})(\RR)_{d,\bm{a}}$
can be identified with a triple $(L_\bullet,\phi,\psi)$ that satisfies certain conditions,
up to isomorphism. In fact, the map $L_\bullet\rar L_\bullet$ of multiplication 
by $\lambda\in\CC^*$
induces an isomorphism of $(L_\bullet,\phi,\psi)$ 
with $(L_\bullet,\lambda^{-2}\phi,\lambda^2\psi)$.
Hence, such a point can be identified with the triple $(L_\bullet,Q,-\phi\psi)$,
where $Q$ is an effective divisor in the linear system
$|\Det \floor{L_\bullet^{-2}} K(P)|$ and
$-\phi\psi=\det(\Phi)\in H^0(S,K^2(2P))$.
Moreover, $L^{-2}_\bullet$ can be reconstructed up to isomorphism from $Q$: so, given
$(Q,\det(\Phi))$, there are exactly $2^{2g}$ choices for $L_\bullet$ and the
set of such choices is a $\Pic^0(S)[2]$-torsor.

Consider then the residues of $\det(\Phi)$. 
\begin{itemize}
\item
Suppose that $i\notin J_{deg}$ and so $\bm{w}$ is non-degenerate at $p_i$.\\
Then necessarily $\Res_{p_i}(\Phi)=0$.
Thus, we will assume that $\bcoa_i=\{0\}$ for all $i\notin J_{deg}$.
\item
Suppose that $j\in J_{deg}$ and so $\bm{w}$ is degenerate at $p_j$.\\
Then the condition on the residue at $p_j$ is not automatically satisfied. 
For $j$ in $J_{nil}$ or $J_{inv}$,
the elements in $\ol{\bcoa}_j$
are detected by their determinant $\det(\bcoa_j)$: thus it is enough to require that $(-\phi\psi)(p_j)=\det(\coa_j)$.
For $j\in J_0$, we must require that $\phi(p_j)=0$ and that $\ord_{p_j}(\phi\psi)>\ord_{p_j}(\phi)$.
\end{itemize}

Consider the space $\Xcal=\Sym^{m}(S\setminus P_{inv})\times H^0(S,K^2(P+P_{deg}))$, where 
$P_{deg}=\sum_{j\in J_{deg}} p_j$ and $P_{inv}=\sum_{j\in J_{inv}} p_j$, and the loci
\begin{align*}
\Qcal & =\{(Q,q)\in\Xcal\,|\,Q\leq \mathrm{div}(q)\}\\
\ol{\Rcal}_i & =\{(Q,q)\in\Xcal\,|\,
\Res_{p_i}(q)=\det(\coa_i) \} & \text{for all $i=1,\dots,n$.}
\end{align*}
Moreover, for $j\in J_0\cup J_{nil}$, the locus $\ol{\Rcal}_j$ can be split into
\begin{align*}
\ol{\Rcal}^-_j & = \{(Q,q)\in\Xcal\,|\, p_j\in Q\} & \text{corresponding to $\phi(p_j)=0$}\\
\ol{\Rcal}^+_j & = \{(Q,q)\in\Xcal\,|\, \ord_{p_j}(q)>\mathrm{mult}_{p_j}(Q)\}
& \text{corresponding to $\psi(p_j)=0$}
\end{align*}
and we call $\Rcal^{0}_j:=\ol{\Rcal}^+_j\cap\ol{\Rcal}^-_j$ for all $j\in  J_0$.
Finally, for every $\bm{\e}:J_{nil}\rar \{+,-\}$ we denote by
$P_\pm(\bm{\e})$ the subsets of points $p_j$ such that $\bm{\e}(j)=\pm$
and let $s_\pm(\bm{\e})=\# P_\pm(\bm{\e})$.
Then we define
\[
\ol{\Rcal}^{\bm{\e}}:=
\left(\bigcap_{j\in J_{inv}}\ol{\Rcal}_j\right)
\cap
\left(
\bigcap_{j\in J_{nil}}
\ol{\Rcal}_j^{\e_j}
\right)
\cap
\left(
\bigcap_{j\in J_0}
\Rcal_j^0
\right)
\quad
\text{and} \quad \ol{\Rcal}=\bigcup_{\bm{\e}} \ol{\Rcal}^{\bm{\e}}.
\]
and we call $\Rcal^{\bm{\e}}:=\ol{\Rcal}^{\bm{\e}}\setminus \bigcup_{j\in J_{nil}} \Rcal_j^0$
and $\Rcal:=\bigcup_{\bm{\e}}\Rcal^{\bm{\e}}$.

The locally closed subvariety $\Qcal$ has codimension $m$ in $\Xcal$ and it is isomorphic
to a holomorphic vector bundle of rank $m'$ over
$\Sym^{m}(S\setminus P_{inv})$ and so the above discussion leads to the following
conclusion.

\begin{proposition}[Topology of $\sigma$-fixed components]\label{prop:topology}
The loci in $\Xcal$ defined above satisfy the following properties.
\begin{itemize}
\item[(a)]
The locus
$\Qcal\cap\ol{\Rcal}^{\bm{\e}}$ is a holomorphic affine
bundle of rank $m'-[s_{inv}+s_0+s_+(\bm{\e})]$ over $\mathrm{Sym}^{m-[s_0+s_-(\bm{\e})]}(S\setminus P_{inv})$.
The locus $\Qcal\cap\Rcal^{\bm{\e}}$ is obtained from $\Qcal\cap\ol{\Rcal}^{\bm{\e}}$
by first restricting the affine
bundle over $\mathrm{Sym}^{m-[s_0+s_-(\bm{\e})]}\left(S\setminus (P_+(\bm{\e})\cup P_{inv})\right)$
and then removing $s_-(\bm{\e})$ affine subbundles of codimension $1$.
\item[(b)]
The locus $\Qcal\cap\ol{\Rcal}$ is connected, has pure codimension $s+s_0+m$ in $\Xcal$
and consists of the irreducible components
$\Qcal\cap \ol{\Rcal}^{\bm{\e}}$ of dimension $3g-3+n-s_0$.
The locus $\Qcal\cap\Rcal$ is the disjoint union of all $\Qcal\cap\Rcal^{\bm{\e}}$.
\item[(c)]
The morphism
\[
\xymatrix@R=0in{
\Higgs^s(S,2,\bm{w},\Det,\ol{\bCoa})(\RR)_{d,\bm{a}}\ar[rr]&& \Xcal\\
(E_\bullet,\eta,\Phi)
\ar@{|->}[rr] &&
(\mathrm{div}(\phi),\det(\Phi))
}
\]
is a $\Pic^0(S)[2]$-torsor over $\Qcal\cap\ol{\Rcal}$ and 
$\Higgs^s(S,2,\bm{w},\Det,\bCoa)(\RR)_{d,\bm{a}}$
is a $\Pic^0(S)[2]$-torsor over $\Qcal\cap\Rcal$.
\item[(d)]
The restriction $\Higgs^s(S,2,\bm{w},\Det,\ol{\bCoa})(\RR)^{\bm{\e}}_{d,\bm{a}}$
of the $\Pic^0(S)[2]$-torsor 
in (c) over the component $\Qcal\cap \Rcal^{\bm{\e}}$ 
is connected, unless $\Qcal\cap\ol{\Rcal}^{\bm{\e}}$ is an affine space
(i.e. $m-[s_0+s_-(\bm{\e})]=0$): in this case it is necessarily trivial.
\end{itemize}
\end{proposition}

We stress that, by definition, $\Qcal\cap \Rcal^{\bm{\e}}=\emptyset$ if $m-[s_0+s_-(\bm{\e})]<0$ 
or $m'-[s_{inv}+s_0+s_+(\bm{\e})]<0$.

\begin{corollary}[Topology of {$\Pic^0(S)[2]$}-quotient of $\sigma$-fixed irreducible components]\label{cor:quotient-topology}
The quotient $\Higgs^s(S,2,\bm{w},\Det,\ol{\bCoa})(\RR)_{d,\bm{a}}^{\bm{\e}}/\Pic^0(S)[2]$
is isomorphic to a holomorphic affine bundle of rank
$m'-[s-s_-(\bm{\e})]$ over $\Sym^{m-[s_0+s_-(\bm{\e})]}(S\setminus P_{inv})$
and $\Higgs^s(S,2,\bm{w},\Det,\bCoa)(\RR)_{d,\bm{a}}^{\bm{\e}}/\Pic^0(S)[2]$
is isomorphic to the complement of $s_-(\bm{\e})$ affine codimension $1$ subbundles
inside a holomorphic affine bundle of rank
$m'-[s-s_-(\bm{\e})]$ over $\Sym^{m-[s_0+s_-(\bm{\e})]}\left(S\setminus (P_+(\bm{\e})\cup P_{inv})\right)$.
\end{corollary}

\begin{proof}[Proof of Proposition \ref{prop:topology}]
Considering the above discussion, we are only left to prove (d).

The action of $\Pic^0(S)[2]$ on
$\Higgs^s(S,2,\bm{w},\Det,\ol{\bCoa})(\RR)^{\bm{\e}}_{d,\bm{a}}$ is given by
$A\cdot (E_\bullet,\eta,\Phi)\mapsto (E_\bullet\otimes A,\eta,\Phi)$ for $A\in\Pic^0(S)[2]$.

If $m-[s_0+s_-(\bm{\e})]=0$, then $\Qcal\cap\ol{\Rcal}^{\bm{\e}}$ is an affine space and so the torsor is trivial.

Assume now $m-[s_0+s_-(\bm{\e})]>0$ and fix $A\in\Pic^0(S)[2]$.
We want to show that every $(E_\bullet,\eta,\Phi)$
in $\Higgs^s(S,2,\bm{w},\Det,\ol{\bCoa})(\RR)^{\bm{\e}}_{d,\bm{a}}$
can be connected to $(E_\bullet\otimes A,\eta,\Phi)$ by a continuous path.

Consider the map $f:\Higgs^s(S,2,\bm{w},\Det;\bCoa)(\RR)^{\bm{\e}}_{d,\bm{a}}
\rar \Pic^d(S)\times \Sym^{m-s_0(\bm{\e})-s_-(\bm{\e})}(S\setminus P_{inv})$
defined by $f(E_\bullet,\eta,\Phi)=(L,[\phi])$,
where
$E_\bullet=(L^\vee_\bullet\otimes\Det)\oplus L_\bullet$ and
$\Phi=\left(\begin{array}{cc} 0 & \phi \\ \psi & 0\end{array}\right)$.
Since the knowledge of $\bm{w}$ and $\bm{a}$ allows to reconstruct
the parabolic bundle $L_\bullet$ out of $L$,
the image $\widetilde{\mathcal{S}}$ of $f$ can be identified to the locus of couples $(L,Q)$ such that 
$Q_{fix}+Q\in|\Det L^{-2}_\bullet K(P)|$, where $Q_{fix}:=\sum_{j\in J_0} p_j+\sum_{\e_j=-}p_j$.
Note that $\widetilde{\mathcal{S}}$
is an \'etale cover  over $\Sym^{m-s_0(\bm{\e})-s_-(\bm{\e})}(S\setminus P_{inv})$
of degree $2^{2g}$ and $f$ is a fibration with fiber $\CC^{m'-[s-s_-(\bm{\e})]}$.
Thus, it is enough to show that $\widetilde{\mathcal{S}}$ is connected.

Now fix $(L,Q)\in\widetilde{\mathcal{S}}$ and let $B:[0,1]\rar \Pic^0(S)$
be a continuous path from $\Ocal_S$ to $A$, so that
$B^{-2}:[0,1]\rar\Pic^0(S)$ is a closed path.
Since $m-[s_0+s_-(\bm{\e})]>0$, we can choose a point $x\in Q$ and 
we can consider the map
\[
\xymatrix@R=0in{
S\setminus P_{inv} \ar[r] & \Pic^0(S)\\
y \ar@{|->}[r] & \Ocal_S(y-x)
}
\]
which induces a surjection $\pi_1(S\setminus P_{inv})\twoheadrightarrow \pi_1(\Pic^0(S))$.
Thus, there exists a path $Y:[0,1]\rar S\setminus P_{inv}$ based at $Y(0)=Y(1)=x$
which is mapped to a path homotopic to $B^{-2}$.
Define $Q(t):=Q-x+Y(t)$ and let $A(t)$ be the unique continuous path in $\Pic^0(S)$
such that $A(0)=\Ocal_S$ and $A(t)^2=\Ocal_S(x-Y(t))$. 
By definition, $Q_{fix}+Q(t)\in |\Det (A(t)\otimes L_\bullet)^{-2}K(P)|$.
Since the path $A^{-2}$
is homotopic to $B^{-2}$, we have $A(1)=B(1)=A$ and so the path
$t\mapsto (L\otimes A(t),Q(t))$ joins $(L,Q)$ and $(L\otimes A,Q)$.
\end{proof}

Provided $e(d,\bm{a},\bm{w})$ remains positive,
Proposition \ref{prop:topology} shows that the isomorphism class of
the moduli space $\Higgs^s(S,2,\bm{w},\Det,\ol{\bCoa})(\RR)^{\bm{\e}}_{d,\bm{a}}$
remains constant as a parabolic weight $w_1(p_i)$ is varied
within the interval $(0,1/2)$ and $\bcoa_i$ is kept equal to $\{0\}$.
Moreover, if $w_1(p_i)\in (0,1/2)$
is pushed to $w'_1(p_i)=0;\ 1/2$ and the class $\bcoa_i$ is switched
to $\bcoa'_i=\{\text{nilpotents}\}$, then the moduli space is isomorphic to
some $\Higgs^s(S,2,\bm{w'},\Det,\ol{\bCoa}')(\RR)^{\bm{\e'}}_{d',\bm{a'}}$,
where $\e'_i=+$ if $e(d',\bm{a'},\bm{w'})>e(d,\bm{a},\bm{w})>0$, and $\e'_i=-$ if 
$0<e(d',\bm{a'},\bm{w'})<e(d,\bm{a},\bm{w})$.
More precisely, we have the following.

\begin{corollary}[Varying the parabolic weights]\label{cor:moving}
Fix $\bCoa$, $d$, $\bm{a}$ and $\bm{w}$ such that $e=e(d,\bm{a},\bm{w})>0$, and
assume that $\bm{w}$ is non-degenerate at $p_i$ and $\bcoa_i=\{0\}$.
Then $\Higgs^s(S,2,\bm{w},\Det,\ol{\bCoa})(\RR)^{\bm{\e}}_{d,\bm{a}}$ is isomorphic to
\begin{itemize}
\item[(a)]
$\Higgs^s(S,2,\bm{w'},\Det,\ol{\bCoa})(\RR)^{\bm{\e'}}_{d',\bm{a'}}$ 
with $d'=d$, $\bm{a'}=\bm{a}$, $\bm{\e'}=\bm{\e}$ and
for every $\bm{w'}$ 
that differs from $\bm{w}$ only on the $i$-th entrance and
such that $0<w'_1(p_i)<1/2$ (see also Nakajima \cite{nakajima:hyperkaehler});
\item[(b)]
$\Higgs^s(S,2,\bm{w'},\Det,\ol{\bCoa}')(\RR)^{\bm{\e'}}_{d',\bm{a'}}$ where
$\bm{w'},\bCoa',\bm{a'}$ differ from $\bm{w},\bCoa,\bm{a}$ only on the $i$-th entrance,
$\bcoa'_i=\{\text{nilpotents}\}$, $a'_i=0$ and either of the following hold:
\[
\begin{array}{|c|c|c|c|c|}
\hline
a_i & w'_1(p_i) & \e'_i & d' & e'\\
\hline\hline
0 & 0 & - & d & e-2w_1(p_i)\\
\hline
0 & 1/2 & + & d & e+(1-2w_1(p_i))\\
\hline
1 & 0 & + & d+1 & e+2w_1(p_i)\\
\hline
1 & 1/2 & - & d & e-(1-2w_1(p_i))\\
\hline
\end{array}
\]
\end{itemize}
as long as $e'=e(d',\bm{a'},\bm{w'})>0$.
\end{corollary}

Similarly,
since $\Res_{p_i}(q)=\det(\bcoa_i)$ is an affine equation
in $H^0(S,K^2(P+P_{deg}))$, we also have the following result.

\begin{corollary}[Varying the quadratic residue of $\det(\Phi)$]
Let $(\bm{w},\bCoa),d,\bm{a}$ be such that $e(d,\bm{a},\bm{w})>0$
and assume that $\det(\bcoa_i)>0$.
Then $\Higgs^s(S,2,\bm{w},\Det,\ol{\bCoa})^{\bm{\e}}_{d,\bm{\a}}$
is isomorphic to $\Higgs^s(S,2,\bm{w},\Det,\ol{\bCoa}')^{\bm{\e}}_{d,\bm{\a}}$,
where $\bCoa'$ differs from $\bCoa$ only on the $i$-th entrance 
and $\det(\bcoa'_{i})>0$.
\end{corollary}

Notice that compact components may also occur, but only in a few limited cases.
If $g=0$ and $n=3+s_0$, then $\Higgs^s(S,2,\bm{w},\Det,\ol{\bCoa})(\RR)^{\bm{\e}}_{d,\bm{a}}$
consists of a single point (if it is nonempty) and so it is compact. 
Biquard-Tholozan remarked that the other cases of compact components with non-degenerate parabolic
structure correspond to representations that Deroin-Tholozan \cite{deroin-tholozan:super-maximal}
call ``super-maximal'' via Theorem \ref{thm:correspondence}.
I would like to thank Nicolas Tholozan for drawing my attention to this point.

\begin{corollary}[Compact components]\label{cor:supermaximal}
Assume $e=e(d,\bm{a},\bm{w})>0$ and $(g,n)\neq (0,3+s_0)$. The locus
$\Higgs^s(S,2,\bm{w},\Det,\ol{\bCoa})(\RR)^{\bm{\e}}_{d,\bm{a}}$ is compact
if and only if
\[
(\star) \begin{cases}
g=0\\
s_-(\bm{\e})=s_{inv}=0\\
\dis e(d,\bm{a},\bm{w})=1-\sum_{\substack{i\notin J_{deg} \\ a_i=0}} (1-2w_1(p_i)) -
\sum_{\substack{i\notin J_{deg}\\ a_i=1}} 2w_1(p_i)\in (0,1].
%
%
\end{cases}
\]
In this case, Higgs bundles $[E,\Phi]\in \Higgs^s(S,2,\bm{w},\Det,\ol{\bCoa})(\RR)^{\bm{\e}}_{d,\bm{a}}$
have $\psi=0$ and so nilpotent $\Phi$, and
the whole component is isomorphic to $\CC\PP^{n-3-s_0}$.\\
The locus $\Higgs^s(S,2,\bm{w},\Det,\bCoa)(\RR)^{\bm{\e}}_{d,\bm{a}}$ is compact
if and only if the condition $(\star)$ is satisfied and $s_+(\bm{\e})=0$.
Moreover, in this case such component is again isomorphic to $\CC\PP^{n-3-s_0}$.
\end{corollary}
\begin{proof}
%
Since $e=e(d,\bm{a},\bm{w})=-d_0+2d+\sum_{i=1}^n 2w_L(p_i)$,
we have
\begin{align*}
m'-s+s_-(\bm{\e}) &=(-d_0+2d)+g-1+n+\NO{\bm{a}}+s\odd-s+s_-(\bm{\e})=\\
&= e-\sum_{i=1}^n 2w_L(p_i)+g-1+n+\NO{\bm{a}}-s\even+s_-(\bm{\e})=\\
& = g-1+e+\sum_{i\notin J_{deg}}\left(1+a_i-2w_L(p_i)\right)+s_-(\bm{\e})
\geq e+g-1
\end{align*}
In Corollary \ref{cor:quotient-topology}, $\Higgs^s(S,2,\bm{w},\Det,\ol{\bCoa})(\RR)^{\bm{\e}}_{d,\bm{a}}$
is presented as a fibration over a symmetric product and the fiber is an open subset of an affine space. Thus, it is 
compact if and only if the fiber is $0$-dimensional and the base is 
a symmetric product of a compact surface (i.e. $s_{inv}=0$).
For the fiber to be $0$-dimensional, we must have $m'\leq s-s_-(\bm{\e})$.
Since $1+a_i-2w_L(p_i)> 0$ for $i\notin J_{deg}$ 
and $e>0$, this implies that
$g=0$, $s_-(\bm{\e})=0$ and
\[
0<e=1-\sum_{i\notin J_{deg}}\left(1+a_i-2w_L(p_i)\right)\leq 1.
\]
Vice versa, if the above numerical conditions are satisfied, it is immediate to check
that the fiber is $0$-dimensional and indeed it consists of a single point.
In this case, such component is isomorphic to $\Sym^{n-3-s_0}(\CC\PP^1)\cong \CC\PP^{n-3-s_0}$.

Similarly, the open component $\Higgs^s(S,2,\bm{w},\Det,\bCoa)(\RR)^{\bm{\e}}_{d,\bm{a}}$,
fibers over $\Sym^{n-3-s_0}(S\setminus (P_{inv}\cup P_+))$.
The numerical conditions for the $0$-dimensionality of the fiber are the same; for the base
to be compact we need $s_{inv}=s_+(\bm{\e})=0$. Again
the component will be isomorphic to $\CC\PP^{n-3-s_0}$.
%
%
\end{proof}

%
%

\section{Hitchin-Simpson correspondence and topology}\label{sec:correspondence}

Fix a complex structure $I$ on the compact surface $S$ 
and let $\Ocal_S$ be the sheaf of $I$-holomorphic functions on $S$.

\subsection{Closed case}

Let $S$ be compact and unpunctured ($n=0$).

Since locally constant functions on $S$ are $I$-holomorphic,
a flat $\CC$-vector bundle $V$ on $S$ can be naturally given an $I$-holomorphic structure $\ol{\pa}^V$.
In particular, if $(\xi,\nabla)$ is a flat $\GL_N$-bundle, then 
$V:=\xi\times_{\GL_N}\CC^N$ is a $I$-holomorphic
bundle endowed with a flat connection, which we will still denote by $\nabla$ by a little abuse.

The point of departure is then the following classical result.

\begin{theorem}[Narasimhan-Seshadri \cite{narasimhan-seshadri:unitary}]\label{thm:narasimhan-seshadri}
The map that sends a flat $\U_N$-bundle
$(\xi,\nabla)$ on $S$ to the $I$-holomorphic vector bundle $E:=\xi\times_{\U_N}\Ocal_S^{\oplus N}$ induces a real-analytic
homeomorphism
\[
\xymatrix@R=0in{
\Flat^{irr}(S,\U_N) \ar[rr]^{\sim}&& \Bun^s(S,N)_0
}
\]
between the moduli space of irreducible flat $\U_N$-principal bundles on $S$
and the moduli space of stable $I$-holomorphic vector bundles of rank $N$ and degree $0$ on $(S,I)$.
Such homeomorphism restricts to 
\[
\xymatrix@R=0in{
\Flat^{irr}(S,\SU_N)\ar[rr]^{\sim}&&\Bun^s(S,N,\Ocal_S)\\
(\xi,\nabla)\ar@{|->}[rr] && (E,\eta)
}
\]
where $E=\xi\times_{\SU_N}\Ocal_S^{\otimes N}$ and $\eta:\det(E)\arr{\sim}\Ocal_S$
sends the unit volume element to $1$.
\end{theorem}

Such correspondence can be lifted to bundles endowed with a trivialization $\tau$ at the base-point $b$;
moreover, just by taking direct sums it can be extended to a correspondence
between completely decomposable $b$-framed flat bundles and polystable $b$-framed $I$-holomorphic bundles.

\begin{corollary}
There is a real-analytic homeomorphism
\[
\xymatrix@R=0in{
\Flat_b^{dec}(S,\U_N)\ar[rr]^{\sim} && \Bun^{ps}_b(S,N)_0\\
(\xi,\nabla,\tau) \ar@{|->}[rr] && (E=\xi\times_{\U_N} \Ocal_S^{\oplus N},\tau'=\tau\otimes_{\U_N}\CC^{N})
}
\]
which restricts to
\[
\xymatrix@R=0in{
\Flat_b^{dec}(S,\SU_N)\ar[rr]^{\sim} && \Bun^{ps}_b(S,N,\Ocal_S)\\
(\xi,\nabla,\tau) \ar@{|->}[rr] && (E=\xi\times_{\SU_N} \Ocal_S^{\oplus N},\eta,\tau').
}
\]
\end{corollary}

In the proof by Narasiman-Seshadri, surjectivity is achieved by continuity method. 
In particular, it does not provide a way to construct $(\xi,\nabla)$ starting from a stable $E$.
The following important result fills such a gap.

\begin{theorem}[Donaldson \cite{donaldson:narasimhan-seshadri}]\label{thm:donaldson-NS}
A holomorphic vector bundle $E$ of rank $N$ and degree $0$ on $(S,I)$
admits a flat invariant metric if and only if $E$ is polystable. Moreover, such a metric is unique
up to automorphisms of $E$.
\end{theorem}

In the case of bundles with monodromy not contained in $\U_N$, no invariant Hermitian metric is available.
In order to codify all information in term of holomorphic structures on $(S,I)$,
the idea is to replace invariant metrics by harmonic metrics.

\begin{definition}
A {\it{harmonic metric}} on a flat $\GL_N$-bundle $(\xi,\nabla)$ on $(S,I)$
is the Hermitian metric $H=h\,h^T$ on the flat bundle $V=\xi\times_{\GL_N}\CC^N$
associated to a section $h:S\rar \xi\times_{\GL_N}(\GL_N/\U_N)$ that minimizes
the energy with respect to the natural metric on the symmetric space $\GL_N/\U_N$.
\end{definition}

Existence and uniqueness of harmonic metrics
and was first proven by Donaldson \cite{donaldson:twisted} in the rank $2$ case.
A more general existence theorem is due to Corlette: here we recall the statement
for Riemann surfaces only.

\begin{theorem}[Corlette \cite{corlette:harmonic}]\label{thm:corlette}
A flat $\GL_N$-bundle $(\xi,\nabla)$ on $(S,I)$ has a harmonic metric $H$
if and only if $\xi$ is reductive (i.e. the closure of the image of $\hol_\xi$ in $\GL_N(\CC)$ is reductive). \\
Moreover, such an $H$ is unique up to automorphisms of $\xi$.
\end{theorem}

Given a hermitian metric $H$ on the flat
vector bundle $V=\xi\times_{\GL_N}\CC^N$, 
the connection $\nabla$ decomposes as
\[
\nabla=\nabla^H+\Phi+\ol{\Phi}^H
\]
where $\Phi$ is an $\End(V)$-valued $(1,0)$-form, $\ol{\Phi}^H$ is its $H$-adjoint and
$\nabla^H$ is compatible with $H$.
We can define a holomorphic structure on $E=V$ by
letting $\ol{\pa}^E=\ol{\pa}^V-\ol{\Phi}^H$, so that $\nabla^H$ is a Chern
connection for the holomorphic Hermitian bundle $(E,H)$.

Harmonicity of the metric $H$ is then equivalent to the 
$\ol{\pa}^E\!-$holomorphicity of
$\Phi\in C^\infty(S,K\otimes \End(E))$. Thus, once we find a harmonic metric $H$
on $(V,\nabla)$, we can produce a Higgs bundle $(E,\Phi)$
of degree $\deg(E)=\deg(V)=0$.\\

Conversely, given a Higgs bundle $(E,\Phi)$ on $(S,I)$ and a Hermitian metric $H$ on $E$,
we can consider the underlying complex vector bundle $V$ endowed with the connection
$\nabla=\nabla^H+\Phi+\ol{\Phi}^H$, where $\nabla^H$ is the Chern connection on $(E,H)$
and $\ol{\Phi}^H$ is the $H$-adjoint of $\Phi$. Harmonicity of 
the metric $H$ on $(V,\nabla)$ is equivalent to the flatness of $\nabla$
and the following theorem provides the wished counterpart to Theorem \ref{thm:corlette}
(proven before by Hitchin \cite{hitchin:self-duality} in the rank $2$ case).

\begin{theorem}[Simpson \cite{simpson:yang-mills}]
A $\GL_N$-Higgs bundle $(E,\Phi)$ on $(S,I)$ supports a metric $H$ such that
the induced $(V,\nabla)$ is flat if and only if $(E,\Phi)$ is polystable.
Moreover, such a metric is unique up to automorphisms of $(E,\Phi)$.
\end{theorem}

Considering the $\SL_N$-bundles correspond to vector bundles with trivializable
determinant, we summarize the above results as follows.
%
%

\begin{corollary}[Simpson \cite{simpson:higgs1} \cite{simpson:higgs2}]
There are smooth diffeomorphisms
\[
\xymatrix@R=0in{
\Flat_b^{red}(S,\GL_N)\ar[rr] && \Higgs^{ps}_b(S,N)_0\\
\Flat_b^{red}(S,\SL_N)\ar[rr] && \Higgs^{ps}_b(S,N,\Ocal_S)
}
\]
which induce a correspondence
\[
\xymatrix@R=0in{
\Flat^{Zd}(S,\GL_N)\ar[rr] && \Higgs^{s}(S,N)_0\\
\Flat^{Zd}(S,\SL_N)\ar[rr] && \Higgs^{s}(S,N,\Ocal_S)
}
\]
between the space of flat bundles with Zariski-dense monodromy
and the space of stable $I$-holomorphic Higgs bundles.
\end{corollary}

Again, the case $N=2$ of above corollary is due to Hitchin \cite{hitchin:self-duality}, whose construction also implies that the involved diffeomorphisms are real-analytic.
Simpson showed that in general
the above correspondence does not continuously extend
over the whole semi-stable locus (see \cite{simpson:higgs2}, pp.38--39).\\

We remind that flat $\U_N$-bundles
are always decomposable and so reductive;
in this case, Zariski-density of the monodromy
is equivalent to irreducibility.

\subsection{Punctured case}\label{sec:correspondence-punctured}


Given a flat $\U_N$-bundle $(\xi,\nabla)$ on the punctured surface
$\dot{S}$, we can as before
produce a $I$-holomorphic vector bundle $\dot{E}=\xi\times_{\U_N}\Ocal_{\dot{S}}^{\oplus N}$ on $\dot{S}$
which carries an invariant Hermitian metric $H$.
Moreover, such $\dot{E}$ admits a unique extension $E\rar S$
{\it{(Deligne extension \cite{deligne:equations})}}
such that the induced $\nabla$ has real residues at $p_i$ with eigenvalues
$0\leq w_1(p_i)<w_2(p_i)<\dots<w_{b_i}(p_i)<1$ for all $p_i\in P$
and algebraic multiplicities $m_1(p_i),\dots,m_{b_i}(p_i)$.

\begin{notation}
Given an endomorphism $f$ of a complex vector space and a $w\in\RR$,
we denote by $\Eig_w(f)$ the direct sum of generalized eigenspaces of $f$
corresponding to eigenvalues with real part $w$.
We also denote by $\Eig_{\geq w}(f)$ the direct sum of all $\Eig_{w'}(f)$
with $w'\geq w$. 
%
\end{notation}

Equipping each vector space $E|_{p_i}$ with the flag
\[
E|_{p_i}=\Eig_{\geq w_1(p_i)}(\Res_{p_i}(\nabla))\supsetneq \Eig_{\geq w_2(p_i)}(\Res_{p_i}(\nabla))\supsetneq\dots\supsetneq \Eig_{\geq w_{b_i}(p_i)}(\Res_{p_i}(\nabla))\supsetneq\{0\}
\]
defines a parabolic structure on $E$ at $P$ of type $\bm{w}$.

If a local section $s$ near $p_i$ satisfies
$0\neq s(p_i)\in\Eig_w(\Res_{p_i}(\nabla))$, then $\ord_{p_i}\|s\|_H=w$ and so
the parabolic structure just defined corresponds to the filtration
\[
E_w =
\left\{s\ |\ \|s\|_H\cdot |z_i|^{\e-w}\ \text{bounded near $p_i$ for all $\e>0$ and all $p_i\in P$}
\right\}\subset E(\infty\cdot P)
\]
where $z_i$ is a local holomorphic coordinate on $S$ centered at $p_i$.

Thus, 
$\Res_{p_i}(\nabla)$ belongs to the conjugacy class $\CL{M_i}\subset \u_N$ of
\[
M_i=\left(
\begin{array}{c|c|c}
w_1(M_i)\id_{m_1(M_i)} & 0 & 0\\
\hline
0 & \ddots & 0\\
\hline
0 & 0 & w_{b_i}(M_i)\id_{m_{b_i}(M_i)}
\end{array}
\right)
\]
where $w_k(M_i)=w_k(p_i)$ and $m_k(M_i)=m_k(p_i)$,
%
%
and the monodromy of $\nabla$ along $\pa_i$ belongs to the conjugacy class $\bco_i=\CL{\exp[-2\pi\sqrt{-1}M_i]}\subset\U_N$
of $\exp(-2\pi \sqrt{-1}M_i)$.

\begin{notation}
If $\bco_i=\CL{C_i}\subset \U_N$, then there exists a unique matrix $M_i\in \sqrt{-1}\cdot\u_N$
with real eigenvalues in $[0,1)$ such that $\exp(-2\pi\sqrt{-1}M_i)=C_i$.
We will write $w_k(\bco_i):=w_k(M_i)$ and $m_k(\bco_i):=m_k(M_i)$ and $\bm{w}(\bCo)$ for the collection of
all $w_k(M_i)$ and $m_k(M_i)$.
%
\end{notation}

We can now state the analogue
of Theorem \ref{thm:narasimhan-seshadri} for punctured surfaces, whose proof
is again by continuity method.

\begin{theorem}[Mehta-Seshadri \cite{mehta-seshadri:parabolic}]\label{thm:mehta-seshadri}
The map that sends a flat $\U_N$-bundle
$(\xi,\nabla)$ on $\dot{S}$
to the Deligne extension $E_\bullet$ of
the $I$-holomorphic vector bundle $\dot{E}=\xi\times_{\U_N}\Ocal_S^{\oplus N}$ 
induces a real-analytic homeomorphism
\[
\xymatrix@R=0in{
\Flat^{irr}(\dot{S},\U_N,\bCo) \ar[rr]^{\sim}&& \Bun^s(S,\bm{w}(\bCo),N)_0
}
\]
between the moduli space of irreducible flat $\U_N$-principal bundles on $\dot{S}$
with monodromy along $\pa_i$ in $\bco_i$
and the moduli space of stable $I$-holomorphic parabolic vector bundles of rank $N$, type $\bm{w}(\bCo)$
and degree $0$ on $(S,I)$.

Such homeomorphism restricts to 
\[
\xymatrix@R=0in{
\Flat^{irr}(\dot{S},\SU_N,\bCo)\ar[rr]^{\sim}&&\Bun^s(S,\bm{w}(\bCo),N,\Ocal_S)\\
(\xi,\nabla)\ar@{|->}[rr] && (E_\bullet,\eta)
}
\]
where $\eta:\det(E_\bullet)\rar\Ocal_S$ sends the unit volume element to $1$.
\end{theorem}

As seen before, the statement extends to $b$-framed polystable parabolic bundles.

\begin{corollary}
There are real-analytic homeomorphisms
\[
\xymatrix@R=0in{
\Flat_b^{dec}(\dot{S},\U_N,\bCo)\ar[rr]^{\sim} && \Bun^{ps}_b(S,\bm{w}(\bCo),N)_0\\
\Flat_b^{dec}(\dot{S},\SU_N,\bCo)\ar[rr]^{\sim} && \Bun^{ps}_b(S,\bm{w}(\bCo),N,\Ocal_S).
}
\]
In both cases, irreducible flat bundles correspond to stable $I$-holomorphic bundles.
\end{corollary}

The counterpart to Theorem \ref{thm:donaldson-NS}
for punctured surfaces was proven by Biquard using analytic methods.

\begin{theorem}[Biquard \cite{biquard:paraboliques}]
A holomorphic vector bundle $E_\bullet$ of rank $N$ and degree $0$ on $(S,I)$
with parabolic structure at $P$ of type $\bm{w}$
admits a flat invariant metric if and only if $E_\bullet$ is polystable. Moreover, such a metric is unique
up to automorphisms of $E_\bullet$.
\end{theorem}

The above achievements (both in the closed and punctured case)
culminate in the more general correspondence between
flat $\GL_N$-bundles on punctured surfaces and parabolic Higgs bundles proven by Simpson \cite{simpson:harmonic}.
Here we describe how one direction works, namely how to go
from flat bundles to parabolic Higgs bundles.
%

Given a $\GL_N$-bundle $\xi$ on $\dot{S}$,
the induced flat vector bundle $\dot{V}=\xi\times_{\GL_N}\CC^N$ can be extended
to $V\rar S$ in such a way that the real parts of the eigenvalues $\lambda_k(p_i)+i\nu_k(p_i)$ of $\Res_{p_i}(\nabla)$ satisfy $0\leq \lambda_1(p_i)<\dots<\lambda_{b_i}(p_i)<1$.
Moreover, if $\xi$ has reductive monodromy, an adaptation of
Corlette's theorem in the noncompact case (see also Labourie \cite{labourie:harmonic})
ensure the existence (and uniqueness up to isomorphism) of
a {\it{tame}} harmonic metric $H$ on the flat vector bundle $\dot{V}=\xi\times_{\GL_N}\CC^N$.
This means that $\nabla$ on $E:=V$ can be decomposed as $\nabla=\nabla^H+\Phi+\ol{\Phi}^H$,
where $\nabla^H$ is compatible with $H$ as before, $\Phi$ and its $H$-adjoint $\ol{\Phi}^H$
have {\it{at worst simple poles at $P$}}
and the {\it{Higgs field $\Phi\in C^\infty(S,K(P)\otimes\End(E_\bullet))$}} is
holomorphic with respect to $\ol{\pa}^E=\ol{\pa}^V-\ol{\Phi}^H$.

Furthermore, $E$ can be endowed with a parabolic structure at $P$ of type $\bm{w}$ 
defined by the filtration
\[
E|_{p_i}=\Eig_{\geq w_1(p_i)}(\Res_{p_i}(\nabla))\supsetneq \Eig_{\geq w_2(p_i)}(\Res_{p_i}(\lambda))\supsetneq\dots\supsetneq \Eig_{\geq w_{b_i}(p_i)}(\Res_{p_i}(\nabla))\supsetneq\{0\}
\]
where $w_k(p_i)=\lambda_k(p_i)$ and $m_k(p_i)=\dim\,\Eig_{\lambda_k(p_i)}(\Res_{p_i}(\nabla))$.

\begin{notation}
Given a matrix $M=D+2M^0\in\gl_N(\CC)$ in Jordan form, with $D$ diagonal and $M^0$ nilpotent,
we call $M'=\mathrm{Re}(D)+M^0$ and $M''=\sqrt{-1}\,\mathrm{Im}(D)+M^0$.
\end{notation}

It can be checked that, if $\Res_{p_i}(\nabla)$ belongs to $\CL{M_i}\subset \psl_N$
with $M_i$ in Jordan form
and so the monodromy along $\pa_i$ belongs to $\bco_i=\CL{\exp(-2\pi\sqrt{-1}M_i)}$,
then $\Res_{p_i}(\Phi)$ belongs to $\bcoa_i=\CL{M''_i/2}\subset\psl_N$.
As before, we will denote by $\bm{w}(\bCo)$ the collection of all $w_k(M_i)$ and $m_k(M_i)$.\\

Summarizing our discussion, the correspondence preserves generalized eigenspaces
of $\hol_\xi(\gamma_i)$, of $\Res_{p_i}(\nabla)$ and of $\Res_{p_i}(\Phi)$;
inside a single generalized eigenspace it works as illustrated in the table below 
(borrowed from \cite{simpson:harmonic}), 
where $\varsigma$ is 
a local holomorphic section of $E$ that does not vanish at $p_i$.
\[
\begin{array}{|c|c|c|}
\hline
& (E_\bullet,\Phi) & (V,\nabla) 
\\
\hline
\text{jump at $p_i$} &
\lambda & \lambda
\\
\hline
\text{residue eigenvalue at $p_i$} &
\sqrt{-1}\,\nu/2 & \lambda+\sqrt{-1}\,\nu \\
\hline
\end{array}
\begin{array}{|c|c|}
\hline
 \text{monodromy} & \ord_{p_i} \|\varsigma\|_H \\
 \hline
 & \lambda\\
\exp\big[-2\pi \sqrt{-1}(\lambda+\sqrt{-1}\,\nu)\big]
& 
\\
\hline
\end{array}
\]
\begin{remark}
The order of growth of $\|\varsigma\|_H$ near $p_i$
may have logarithmic factors as in Example \ref{example:log} (see \cite{simpson:harmonic}).
More precisely, if $\varsigma(p_i)$ takes values in a subspace of $V|_{p_i}$
corresponding to a Jordan block of $\Res_{p_i}(\nabla)$ of size $m$ and eigenvalue
$\lambda+\sqrt{-1}\,\nu$, then
\[
\|\varsigma\|_H\sim |z_i|^{\lambda}|\log|z_i||^{l-\frac{m+1}{2}}
\]
where $z_i$ is a local coordinate on $S$ at $p_i$ and $l>0$ is the smallest integer such that $\left[\Res_{p_i}(\nabla)-\left(\lambda+\sqrt{-1}\,\nu\right)\id\right]^l s(p_i)=0$.
\end{remark}

The above set-theoretic correspondence can be promoted to a real-analytic one. Similarly to the closed case treated by Hitchin and
Simpson, the moduli spaces of flat bundles and of Higgs bundles are different holomorphic manifestations of the same
hyperk\"ahler manifold: the moduli space of harmonic bundles.
This is proven by Konno \cite{konno:construction} for
Higgs fields with nilpotent residues and by Biquard-Boalch
\cite{biquard-boalch:wild} in general.

\begin{theorem}[Simpson \cite{simpson:harmonic}, Konno \cite{konno:construction}, Biquard-Boalch \cite{biquard-boalch:wild}]\label{thm:simpson-konno}
Let $M_i\in\gl_N$ be a matrix in Jordan form
and let $\bco_i=\CL{\exp(-2\pi\sqrt{-1}M_i)}$ be a conjugacy class in $\GL_N$ and
$\bcoa_i=\CL{M''_i/2}\subset\gl_N$
for $i=1,\dots,n$.\\
There is a real-analytic diffeomorphism
\[
\xymatrix@R=0in{
\Flat_b^{red}(\dot{S},\GL_N,\bCo)\ar[rr]^{\sim} && \Higgs^{ps}_b(S,\bm{w},N,\bCoa)_0\\
}
\]
between the moduli space of $b$-framed flat $\GL_N$-bundles $(\xi,\nabla,\tau)$ on $\dot{S}$ with
reductive monodromy and $\hol_\xi(\pa_i)\in\bco_i$
and the moduli space of $b$-framed polystable Higgs bundles 
$(E_\bullet,\eta,\Phi,\tau)$
on $(S,I)$ of parabolic type $\bm{w}=\bm{w}(\bCo)$ at $P$
and degree $0$ with $\Res_{p_i}(\Phi)\in\bcoa_i$.\\
Under such diffeomorphism, flat bundles with Zariski-dense monodromy correspond to
stable parabolic Higgs bundles, and so the induced
\[
\xymatrix@R=0in{
\Flat^{Zd}(\dot{S},\GL_N,\bCo)\ar[rr]^{\sim} && \Higgs^{s}(S,\bm{w},N,\bCoa)_0\\
}
\]
is a real-analytic diffeomorphism too.
Similarly, via
\[
\xymatrix@R=0in{
\Flat_b^{red}(\dot{S},\SL_N,\bCo)\ar[rr]^{\sim} && \Higgs^{ps}_b(S,\bm{w},N,\Ocal_S,\bCoa)\\
\Flat^{Zd}(\dot{S},\SL_N,\bCo)\ar[rr]^{\sim} && \Higgs^{s}(S,\bm{w},N,\Ocal_S,\bCoa)\\
}
\]
$\SL_N$-bundles correspond to parabolic Higgs bundles $(E_\bullet,\Phi)$ endowed with a trivialization
$\eta:\det(E_\bullet)\arr{\sim}\Ocal_S$.
\end{theorem}
%

\subsection{Correspondence and the real locus in rank $2$}

Following Hitchin \cite{hitchin:self-duality}, consider a point
$(E_\bullet,\eta,\Phi)$ in $\Higgs^s(S,\bm{w},2,\Ocal_S,\bCoa)$
fixed by the involution $\sigma$ that sends $\Phi$ to $-\Phi$.

By Lemma \ref{lemma:real}, $(E_\bullet,\eta,\Phi)$ may belong to
$\Bun^s(S,\bm{w},2,\Ocal_S)$ or to some $\Higgs^s(S,\bm{w},2,\Ocal_S,\bCoa)(\RR)_{d,\bm{a}}$.

In the former case, $\Phi=0$ and so the point corresponds to a flat $\SU_2$-bundle
by Theorem \ref{thm:mehta-seshadri}.

In the latter case, $\Phi\neq 0$ and we assume $e(d,\bm{a},\bm{w})>0$.
Then Lemma \ref{lemma:real}
provides a splitting $E_\bullet=L^\vee_\bullet\oplus L_\bullet$
and an automorphism $\iota$ of $E_\bullet$ that preserves the splitting
and that sends $\Phi$ to $-\Phi$.
Moreover, such splitting and $\iota$ are essentially unique, since $(E_\bullet,\Phi)$ is stable.
Let $\dot{V}:=\dot{E}$ and let $\nabla$ (resp. $\nabla'$) be the flat connection
on $\dot{V}$ associated to $(E_\bullet,\Phi)$ (resp. $(E_\bullet,-\Phi)$).
It is easy to see that the flat vector bundles $(\dot{V},\nabla)$ and $(\dot{V},\nabla')$ support the same
tame harmonic metric $H$, which means that $\nabla=\nabla^H+\Phi+\ol{\Phi}^H$ and
$\nabla'=\nabla^H-\Phi-\ol{\Phi}^H$.
Since $\iota^*(-\Phi)=\Phi$ and $\iota^*(\nabla')=\nabla$, the automorphism $\iota$ 
preserves $\nabla^H$ and so is an $H$-isometry; as a consequence, $L_\bullet$ and
$L^\vee_\bullet$ are $H$-orthogonal and $H$ induces the identification $\ol{L}_\bullet \cong L^\vee_\bullet$.
Since the anti-linear involution
\[
T:\xymatrix@R=0in{
L^\vee_\bullet\oplus L_\bullet \ar[r] & L^\vee_\bullet\oplus L_\bullet\\
(\alpha,\beta)\ar@{|->}[r] & (\ol{\beta},\ol{\alpha})
}
\]
satisfies $T\circ\Phi=\ol{\Phi}^H\circ T$ and so commutes with $\nabla$,
the monodromy of $\nabla$ 
preserves $\dot{V}(\RR):=\mathrm{Fix}(T)\subset\dot{V}$ and so it
defines a representation $\rho_{_\RR}$ that
takes values in $\SL_2(\RR)$.
Moreover, $\dot{V}(\RR)\hra \dot{V}\cong \dot{L}\oplus \dot{\ol{L}}\cong 
\dot{V}(\RR)\otimes\CC$ and we identify $\dot{V}(\RR)$ to $\dot{L}$.

As a consequence, if $\rho_{_\RR}(\gamma_i)$ is elliptic
and $\{\rot\}(\rho_{_\RR}(\gamma_i))=\{r_i\}$,
then 
$\{r_i\}=\deg_{p_i}\{L^{-2}_\bullet\}$. Notice that the power $2$ appears because $\SL_2(\RR)\rar\PSL_2(\RR)$ is a cover of degree $2$.



A version of the following can be found for instance in
Section 3.6 of \cite{burger-iozzi-wienhard:higher}.

\begin{lemma}[Euler number as a first Chern class]\label{lemma:chern}
The parabolic degree of $L_\bullet$ and the Euler number of $\rho$ satisfy $\Eu(\rho) =2\deg(L_\bullet)=2\deg(L)+2\NO{\bm{w_L}}$.
\end{lemma}

The above discussion then leads to the following result,
the bound on $d$ being a consequence of
Proposition \ref{prop:topology}(a) and Lemma \ref{lemma:chern}.

\begin{theorem}[Correspondence for $\SL_2(\RR)$]\label{thm:correspondence}
For every $i=1,\dots,n$, let
\begin{itemize}
\item
$\bco_i$ be a conjugacy class in $\SL_2(\RR)$
\item
$\bcoa_i$ be a conjugacy class in $\psl_2(\CC)$
\item
$w_1(p_i)\in[0,1/2]$ and $\e:J_{nil}=\{j\,|\,\bcoa_j\ \text{nilpotent}\}\rar \{+,-\}$
\item
$a_i\in\{0,1\}$
\end{itemize}
that match according to the following table
%
\begin{center}
\[
\hspace{-1cm}\begin{array}{|c|c|c|c|c|}
\hline
& a  & \bCoa & \bCo & w_L\\
\hline\hline
\text{deg.} & 0 & 0 & (-1)^{2w_1}\id &  w_1=0;\frac{1}{2}\\
\hline
\begin{array}{c}\text{deg.}\\ \e=\pm 1\end{array}
& 0 & \left(\begin{array}{cc}0 & 1\\0 & 0\end{array}\right) & (-1)^{2w_1}\left(\begin{array}{cc}1 & 0\\ \e& 1\end{array}\right) & w_1=0;\frac{1}{2}\\
\hline
\begin{array}{c}\text{deg.}\\ \ell>0\end{array}
 & 0 & \left(\begin{array}{cc}\sqrt{-1}\,\ell/8\pi & 0\\0 & -\sqrt{-1}\,\ell/8\pi\end{array}\right)
& (-1)^{2w_1}\left(\begin{array}{cc}\exp(\ell/2) & 0\\0 & \exp(-\ell/2)\end{array}\right) & w_1=0;\frac{1}{2}\\
\hline
\text{non-deg.} & \{0,1\} & 0 & \left(\begin{array}{cc}\cos(2\pi w_L) & \sin(2\pi w_L)\\-\sin(2\pi w_L) & \cos(2\pi w_L)\end{array}\right)
& a+(-1)^a w_1\\
\hline
\end{array}
\]
\captionof{table}{}
\end{center}
%
Let $d\in\ZZ$ such that
\[
-\NO{\bm{w_L}}<d\leq g-1+\frac{s\even-\NO{\bm{a}}-s_0-s_-(\bm{\e})}{2}
\]
where $s\even=\#\{j\,|\,w_1(p_j)=0\}$, $s_0=\#\{i\,|\,\bco_i=\{\id\}\}$ and $s_-(\bm{\e})=\#\{j\in J_{nil}\,|\,\e_j=-\}$.\\
Then there are real-analytic diffeomorphisms
\[
\xymatrix@R=0in{
\Rep(\dot{S},\SL_2(\RR),\bCo)_e\ar[rr]^{\sim} && \Higgs(S,\bm{w},2,\Ocal_S,\bCoa)(\RR)^{\bm{\e}}_{d,\bm{a}}\\
\Rep(\dot{S},\PSL_2(\RR),\bCo)_e\ar[rr]^{\sim\qquad} && \Higgs(S,\bm{w},2,\Ocal_S,\bCoa)(\RR)^{\bm{\e}}_{d,\bm{a}}/\Pic^0(S)[2]\\
}
\]
where $e=2d+2\sum_{i=1}^n w_L(p_i)>0$ is the Euler number of the associated oriented $\RR\PP^1$-bundle. Moreover,
the above maps extend to homeomorphisms
\[
\xymatrix@R=0in{
\Rep(\dot{S},\SL_2(\RR),\ol{\bCo})_e\ar[rr]^{\sim} && \Higgs(S,\bm{w},2,\Ocal_S,\ol{\bCoa})(\RR)^{\bm{\e}}_{d,\bm{a}}\\
\Rep(\dot{S},\PSL_2(\RR),\ol{\bCo})_e\ar[rr]^{\sim\qquad} && \Higgs(S,\bm{w},2,\Ocal_S,\ol{\bCoa})(\RR)^{\bm{\e}}_{d,\bm{a}}/\Pic^0(S)[2]\\
}
\]
for the closures $\ol{\bCo}$ of $\bCo$ and $\ol{\bCoa}$ of $\bCoa$.
\end{theorem}  
\begin{proof}
We are only left to prove the last assertion, namely that the map 
$\Rep(\dot{S},\SL_2(\RR),\ol{\bCo})_e\rar \Higgs(S,\bm{w},2,\Ocal_S,\ol{\bCoa})(\RR)^{\bm{\e}}_{d,\bm{a}}$ is a homeomorphism.
We already know that it is bijective and we want to show that it is
continuous and proper. 

Endow $\dot{S}$ with the unique $I$-conformal hyperbolic metric of finite area
and its universal cover $\widetilde{\dot{S}}$ with the pull-back metric.
Let $S^\circ$ be the compact subsurface obtained from $S$ by removing small open disk
neighborhoods of the marked points and let $\widetilde{S}^\circ$ be its preimage
inside the universal cover $\widetilde{\dot{S}}$ and $F\subset\widetilde{\dot{S}}$ be a fundamental domain for the action of $\pi_1(\dot{S})$.
We recall that, for every equivariant harmonic map $\tilde{h}:\widetilde{\dot{S}}\rar
\HH^2\cong \SL_2(\RR)/\SO_2(\RR)$, we have
\begin{equation}\tag{$\star$}\label{eq}
\left|\nabla \tilde{h}\right|^2\leq C\cdot \mathcal{E}_{S^\circ}(\tilde{h})
\end{equation}
at every point of $\widetilde{S}^\circ\cap F$,
where $C$ depends only on the diameter of $\widetilde{S}^\circ\cap F$
and $\mathcal{E}_{S^\circ}(\tilde{h})$ is the energy of the restriction of
$\tilde{h}$ to $\widetilde{S}^\circ\cap F$.

Consider a sequence $\rho^{(k)}$ of representations in
$\Rep(\dot{S},\SL_2(\RR),\ol{\bCo})_e$ and let $(E^{(k)}_\bullet,\eta^{(k)},\Phi^{(k)})$
be the corresponding Higgs bundle. 
By Lemma \ref{lemma:real}, the bundle $E^{(k)}$
is isomorphic to $(L^{(k)}_\bullet)^\vee\oplus L^{(k)}_\bullet$,
with $L^{(k)}$ of fixed degree $d$.
%
%

Suppose now that $\rho^{(k)}\rar\rho$
and let $(E_\bullet=L^\vee_\bullet\oplus L_\bullet,\eta,\Phi)$ the Higgs bundle determined $\rho$.
By Proposition 2.6.1 of \cite{korevaar-schoen:1}, 
there are $\rho^{(k)}$-equivariant harmonic maps
$\tilde{h}^{(k)}:\widetilde{\dot{S}}\rar \HH^2$
that are locally equi-bounded and equi-Lipschitz. Thus,
$\tilde{h}^{(k)}$ locally Lipschitz converges to the unique
$\rho$-equivariant harmonic map $\tilde{h}$.
Thus, $\ol{\pa}^{L^{(k)}}\rar \ol{\pa}^{L}$ 
and $\Phi^{(k)}\rar \Phi$
uniformly on
the compact subsets of $\dot{S}$.
Since the parabolic structure is fixed, this implies that $L^{(k)}\rar L$
and so $\Phi^{(k)}\rar \Phi$. This proves continuity of the correspondence.

Suppose finally that $\rho^{(k)}$ is divergent.
We want to show that $\|\Phi^{(k)}|_{S^\circ}\|^2$ is divergent. 
By contradiction, up to extracting a subsequence, 
$\mathcal{E}_{S^\circ}(\tilde{h}^{(k)})
=2\|\Phi^{(k)}|_{S^\circ}\|^2$ would be bounded.
By the locally uniform bound \eqref{eq}, the $\tilde{h}^{(k)}|_{\widetilde{S}^\circ\cap F}$ would be
equi-Lipschitz and so $\rho^{(k)}$ would have a convergent subsequence.
\end{proof}


Notice that
$\bco_i$ determines $a_i$, $w_1(p_i)$ and $\bcoa_i$ (and the sign $\e_i$, if $\bcoa_i$ is nilpotent) and vice versa.
Thus we can also draw the following conclusion.

\begin{corollary}[Components of $\PSL_2(\RR)$-representations]\label{cor:euler-components}
Connected components 
of $\Rep(\dot{S},\PSL_2(\RR),\bCo)$ with Euler number 
$e=2d+2\NO{\bm{w_L}}>0$ are classified by the integers $d$
such that $-\NO{\bm{w_L}}<d\leq g-1+\frac{1}{2}\left(s\even-\NO{\bm{a}}-s_0-s_-(\bm{\e})\right)$.
\end{corollary}
%

The topology of $\Higgs(S,\bm{w},2,\Ocal_S,\bCoa)(\RR)_{d,\bm{a}}$ 
and of its quotient by $\Pic^0(S)[2]$ is described in Proposition \ref{prop:topology}.

%

\subsection{Uniformization components}

Let $\bm{\ell}=(\ell_1,\dots,\ell_n)$ with $\ell_i=\sqrt{-1}\th_i$ and $\th_i>0$ for $i=1,\dots,k$ and $\ell_i\geq 0$ for $i=k+1,\dots,n$.
Call $s_{0}=\#\{ i\in\{1,\dots,k\}\,|\, \th_i\in 2\pi\NN_+\}$.


A consequence of the above work is the following result stated
in the introduction.

\begin{corollary}[Topology of uniformization components]\label{cor:uniformization}
Assume $e_{\bm{\ell}}>0$ and consider the monodromy map
\[
\hol\circ\Xi_{\bm{\ell}}:\Ycal(\dot{S},\bm{\ell})\lra \Rep(\dot{S},\PSL_2(\RR),\bCo_{\bm{\ell}})_{e_{\bm{\ell}}}.
\]
The space $\Rep(\dot{S},\PSL_2(\RR),\bCo_{\bm{\ell}})_{e_{\bm{\ell}}}$
is real-analytically diffeomorphic to a holomorphic affine bundle of
rank $3g-3+n-m$ over $\mathrm{Sym}^{m-s_0}(S\setminus \{p_{k+1},\dots,p_n\})$, with 
$m=\sum_{1\leq i\leq k}\floor{\frac{\th_i}{2\pi}}$.
\end{corollary}
\begin{proof}
The monodromy map takes values in $\Rep(\dot{S},\PSL_2(\RR),\bCo_{\bm{\ell}})_{e_{\bm{\ell}}}$
by Proposition \ref{prop:uniformization}.
%
The result then follows from
Theorem \ref{thm:correspondence} and Proposition \ref{prop:topology},
remembering that cusps correspond to positive unipotents.
%
\end{proof}

In \cite{deroin-tholozan:super-maximal} Deroin-Tholozan show that
``super-maximal'' components of $\PSL_2(\RR)$-representations
consist of monodromies of hyperbolic metrics.
Deroin told me that this result can be recovered using 
Corollary \ref{cor:supermaximal} and the analysis carried out in Section \ref{sec:topology}
as follows.

\begin{corollary}\label{cor:geometric}
Fix $\bm{\bCo}=(\bco_1,\dots,\bco_n)$ and $e>0$
and assume $n>3+s_0$.
Then $\Rep(\dot{S},\PSL_2(\RR),\ol{\bCo})_e$
is compact if and only if 
\[
\begin{cases}
g=0\\
\text{no $\bco_i$ is hyperbolic or negative unipotent}\\
e=1-\NO{\{\bm{r}\}}\in (0,1].
\end{cases}
\]
Moreover, in this case every representation in $\Rep(\dot{S},\PSL_2(\RR),\ol{\bCo})_e$
is the monodromy of some hyperbolic metric on $\dot{S}$.
\end{corollary}
\begin{proof}
By Corollary \ref{cor:supermaximal}, representations lying in compact components
cannot have hyperbolic or negative unipotent boundary monodromy.

Fix a complex structure $I$ on $S$.

Let $\Det$ be a line bundle of degree $d_0=0$ if $n-1$ is even, or of degree $d_0=1$ if $n-1$ is odd. Fix $d=(d_0+1-n)/2$.
Let $a_1=\dots=a_n=0$ and let $w_1(p_i)=\frac{1}{2}\left(1-\{r_i\}\right)$.
Finally, pick the conjugacy classes $\bCoa_i$ to be $\{0\}$ if $c_i$ is elliptic, and to be nilpotent if $c_i$ is positive unipotent
(and, in this case, we set $\e_i=+$).

By Theorem \ref{thm:correspondence}, the component
$\Rep(\dot{S},\PSL_2(\RR),\ol{\bCo})_e$ is homeomorphic to
$\Higgs(S,\bm{w},2,\Ocal_S,\ol{\bCoa})_{d,\bm{a}}^{\bm{\e}}$
and so the first claim follows from Corollary \ref{cor:supermaximal}.

As for the last claim, remember that all Higgs bundles
$[E,\Phi]$ in the compact $\Higgs(S,\bm{w},2,\Ocal_S,\ol{\bCoa})_{d,\bm{a}}^{\bm{\e}}$
have $\psi=0$ and so $\Phi$ is nilpotent. 
The harmonic section $h$ of $\xi_\rho\times_{\SL_2(\RR)}(\SL_2(\RR)/\SO_2(\RR))$ constructed by Donaldson
corresponds to a $\rho$-equivariant harmonic map $\tilde{h}:\widetilde{\dot{S}}\rar \HH^2\cong \SL_2(\RR)/\SO_2(\RR)$.
Since the quadratic differential $\det(\Phi)$ pulls back to $\widetilde{\dot{S}}$
to the Hopf differential of $\tilde{h}$, the vanishing $\det(\Phi)=0$ implies that $\tilde{h}$ is conformal.
Hence, the pull-back via $\tilde{h}$ of the hyperbolic metric on $\HH^2$ descends to a conformal hyperbolic metric on $(\dot{S},I)$,
possibly with extra conical points of angles in $2\pi\NN_+$, and with monodromy $\rho$.
\end{proof}

The case $(g,n)=(0,3+s_0)$ of the pair of pants can be done by hands.\\

%
%
%
%
%
%
%
The following result by McOwen \cite{mcowen:metric} and Troyanov
\cite{troyanov:prescribing}
is a version of Koebe's uniformization theorem
\cite{koebe:first} \cite{koebe:second}
for hyperbolic surfaces with conical singularities.

\begin{theorem}[Uniformization with conical singularities]\label{thm:troyanov}
For every $\th_1,\dots,\th_n\geq 0$ such that $e_{\sqrt{-1}\bm{\th}}>0$,
there exists exactly one metric on $\dot{S}$ with conical singularity
of angle $2\pi\th_i$ at $p_i$
(or with a cusp at $p_i$, if $\th_i=0$) and which is $I$-conformal.
%
\end{theorem}

Mimicking Hitchin's computation \cite{hitchin:self-duality} in the case of a closed surface,
we then have the following expected consequence.

\begin{corollary}[Uniformization Higgs bundles]
Assume $k=n$ and let $\delta_i=\frac{\th_i}{2\pi}-1$.
Then the monodromy representation $\hol_h:\pi\rar\PSL_2(\RR)$
of the unique hyperbolic metric $h$ in $\Ycal(\dot{S},\bm{\ell})$ 
with conical singularities of angle $2\pi\th_i$ at $p_i$ which is
{\it{$I$-conformal}} corresponds 
to the $\Pic^0(S)[2]$-equivalence class of
the parabolic Higgs bundle $(E_\bullet,\Phi)$, with
\[
E_\bullet=L_\bullet^\vee\oplus L_\bullet,
\quad
\Phi=\left(
\begin{array}{cc}
0 & 1\\
0 & 0
\end{array}
\right)
\]
where $L_\bullet=B(-\frac{1}{2}\bm{\delta}\cdot P)$ and $B^{2}\cong K_S$,
and so $L_\bullet^{-2}K(P)=T_S(\bm{\delta}\cdot P)\otimes K_S(P)$.
\end{corollary}
\begin{proof}
It is enough to notice that the harmonic metric
the $\RR^2$-bundle $\dot{V}\rar\dot{S}$ with monodromy $\hol_h$
is provided by the (equivariant) developing map
$(\wti{\dot{S}},\wti{h})\rar \HH^2=\SL_2(\RR)/\SO_2(\RR)$ itself,
and then follow Hitchin's computation.
\end{proof}

%
%
%
%
%
%
%
%


\bibliographystyle{amsplain}
\bibliography{parabolic-biblio}

\end{document}